\documentclass[a4paper,11pt]{article}
\usepackage[latin1]{inputenc}
\usepackage[english]{babel}
\usepackage{amsmath}
\usepackage{amsfonts}
\usepackage{amssymb}
\usepackage{epsfig}
\usepackage{amsopn}
\usepackage{amsthm}
\usepackage{color}
\usepackage{graphicx}
\usepackage{subfigure}
\usepackage{enumerate}
\newtheorem{theorem}{Theorem}[section]

\newtheorem{lemma}[theorem]{Lemma}

\newtheorem*{theorem*}{Theorem}
\newtheorem*{lemma*}{Lemma}
\newtheorem*{remark*}{Remark}
\newtheorem*{definition*}{Definition}
\newtheorem*{proposition*}{Proposition}
\newtheorem*{corollary*}{Corollary}
\numberwithin{equation}{section}
\newcommand{\real}{\mathbb{R}}

\let\ced=\c         

\def\qed{\,\unskip\kern 6pt \penalty 500
\raise -2pt\hbox{\vrule \vbox to8pt{\hrule width 6pt
\vfill\hrule}\vrule}\par}
\definecolor{darkblue}{rgb}{0.05, .05, .65}
\definecolor{darkgreen}{rgb}{0.1, .65, .1}
\definecolor{darkred}{rgb}{0.8,0,0}
\newcommand{\beqn}{\begin{equation}}
\newcommand{\eeqn}{\end{equation}}
\newcommand{\bear}{\begin{eqnarray}}
\newcommand{\eear}{\end{eqnarray}}
\newcommand{\bean}{\begin{eqnarray*}}
\newcommand{\eean}{\end{eqnarray*}}
\begin{document}

\title{\huge \bf Existence and multiplicity of blow-up profiles for a quasilinear diffusion equation with source}

\author{
\Large Razvan Gabriel Iagar\,\footnote{Departamento de Matem\'{a}tica
Aplicada, Ciencia e Ingenieria de los Materiales y Tecnologia
Electr\'onica, Universidad Rey Juan Carlos, M\'{o}stoles,
28933, Madrid, Spain, \textit{e-mail:} razvan.iagar@urjc.es},\\
[4pt] \Large Ariel S\'{a}nchez,\footnote{Departamento de Matem\'{a}tica
Aplicada, Ciencia e Ingenieria de los Materiales y Tecnologia
Electr\'onica, Universidad Rey Juan Carlos, M\'{o}stoles,
28933, Madrid, Spain, \textit{e-mail:} ariel.sanchez@urjc.es}\\
[4pt] }
\date{}
\maketitle

\begin{abstract}
We classify radially symmetric self-similar profiles presenting finite time blow-up to the quasilinear diffusion equation with weighted source
$$
u_t=\Delta u^m+|x|^{\sigma}u^p,
$$
posed for $(x,t)\in\real^N\times(0,T)$, $T>0$, in dimension $N\geq1$ and in the range of exponents $-2<\sigma<\infty$, $1<m<p<p_s(\sigma)$, where
$$
p_s(\sigma)=\left\{\begin{array}{ll}\frac{m(N+2\sigma+2)}{N-2}, & N\geq3,\\ +\infty, & N\in\{1,2\},\end{array}\right.
$$
is the renowned Sobolev critical exponent. The most interesting result is the \emph{multiplicity of two different types} of self-similar profiles for $p$ sufficiently close to $m$ and $\sigma$ sufficiently close to zero in dimension $N\geq2$, including \emph{dead-core profiles}. For $\sigma=0$, this answers in dimension $N\geq2$ a question still left open in \cite[Section IV.1.4, pp. 195-196]{S4}, where only multiplicity in dimension $N=1$ had been established. Besides this result, we also prove that, for any $\sigma\in(-2,0)$, $N\geq1$ and $m<p<p_s(\sigma)$ \emph{existence} of at least a self-similar blow-up profile is granted. In strong contrast with the previous results, given any $N\geq1$, $\sigma\geq\sigma^*=(mN+2)/(m-1)$ and $p\in(m,p_s(\sigma))$, \emph{non-existence} of any radially symmetric self-similar profile is proved.
\end{abstract}

\


\smallskip


\section{Introduction}

The aim of this paper is to perform a classification of radially symmetric self-similar blow-up patterns to the following reaction-diffusion equation
\begin{equation}\label{eq1}
u_t=\Delta u^m+|x|^{\sigma}u^p, \qquad (x,t)\in\real^N\times(0,\infty),
\end{equation}
in the range of exponents $m>1$, $\sigma\in(-2,\infty)$ and $m<p<p_s(\sigma)$, where
\begin{equation}\label{Sobolev}
p_s(\sigma)=\left\{\begin{array}{ll}\frac{m(N+2\sigma+2)}{N-2}, & N\geq3,\\ +\infty, & N\in\{1,2\},\end{array}\right.
\end{equation}
is the Sobolev critical exponent. The reason for limiting the current work to the subcritical range $p<p_s(\sigma)$ is the fact that, as it will be readily understood from the proofs and the properties of the dynamical system we work with, the exponent $p=p_s(\sigma)$ introduces a strong bifurcation leading to different behaviors of the self-similar profiles in the supercritical range $p>p_s(\sigma)$, a range that will be addressed in a forthcoming study for $\sigma\neq0$, after the case $\sigma=0$ addressed in \cite{GV97, IS24}.

The most interesting feature of Eq. \eqref{eq1} is the competition between the diffusion term and the weighted source term. On the one hand, the porous medium equation
\begin{equation}\label{PME}
u_t=\Delta u^m
\end{equation}
is a well-established model in nonlinear diffusion, arising also from a number of physical and chemical applications and whose theory is already deeply understood, see for example the monograph \cite{VPME}. In particular, the mass of any integrable solution to \eqref{PME} is conserved along the evolution, while solutions to it, in absence of other influences, evolve towards explicit self-similar solutions known as the Zeldovich-Kompaneets-Barenblatt (ZKB) solutions. On the other hand, the weighted source term $|x|^{\sigma}u^p$ contributes to an increase of the initial $L^1$ norm of any solution, possibly leading to a phenomenon of finite time blow-up. By finite time blow-up we understand that there exists $T\in(0,\infty)$ such that $u(t)\in L^{\infty}(\real^N)$ for $t\in(0,T)$ but $u(T)\not\in L^{\infty}(\real^N)$. Starting from the seminal paper by Fujita \cite{Fu66}, the finite time blow-up and properties of the solutions related to it became the most studied features related to Eq. \eqref{eq1}.

This study had begun with the heat equation with source, that is, the limiting case $m=1$ and $\sigma=0$ in Eq. \eqref{eq1}. The monograph \cite{QS} (and references therein) gives an answer to many of the questions related to the finite time blow-up of its solutions, establishing a strong difference between how finite time blow-up takes place in the subcritical range $1<p<p_s(0)$ and in the supercritical range $p>p_s(0)$. In particular, an important difference is stated with respect to the blow-up rate and profiles: while in the former range blow-up is always of ODE type (that is, led by the ODE $u_t=u^p$) and thus of type I, which means
\begin{equation}\label{type1}
\limsup\limits_{t\to T}(T-t)^{1/(p-1)}\|u(t)\|_{\infty}<\infty,
\end{equation}
in the latter range blow-up of type II occurs, that is, there are solutions for which the limit in the left hand side of \eqref{type1} is equal to infinity, and classifications of them are given in \cite{HV94, MM09, MM11}. Letting $m=1$ but with $\sigma>0$, Mukai and Seki \cite{MS21} constructed an infinite sequence of solutions presenting blow-up of type II with different rates for $p>p_s(\sigma)$. We thus notice that even in simpler models, the Sobolev critical exponent splits the interval $(m,\infty)$ into two ranges with very different properties.

Introducing a weighted source term in Eq. \eqref{eq1} came as a next step in the study, once the blow-up phenomenon has been understood with $\sigma=0$. With either $m=1$ or $m>1$, a number of works such as \cite{BK87, AdB91, Pi97, Pi98, Qi98, Su02} studied qualitative properties of solutions and the phenomenon of finite time blow-up with respect to the initial condition, establishing, among other results, the Fujita type exponent for Eq. \eqref{eq1} as
\begin{equation}\label{Fujita}
p_F(\sigma)=m+\frac{\sigma+2}{N}.
\end{equation}
Indeed, it is shown in \cite{Qi98} that any non-trivial solution blows up in finite time if $m<p\leq p_F(\sigma)$, while an example of global solution in forward self-similar form is constructed for $p>p_F(\sigma)$. Suzuki \cite{Su02} links finite time blow-up to Eq. \eqref{eq1} to the spatial decay of the initial condition $u_0(x)$ as $|x|\to\infty$, introducing thus a new threshold for blow-up usually called as the second critical exponent. Later, Andreucci and Tedeev established in \cite{AT05} blow-up rates for Eq. \eqref{eq1} as a particular case of the more general doubly nonlinear equation, under some technical limitations on the range of $\sigma$.

Another important aspect related to the finite time blow-up is the availability of patterns towards which the solutions evolve near the blow-up time. It has been noticed that type I blow-up for Eq. \eqref{eq1} with $\sigma=0$ (according to the estimate \eqref{type1}) approaches self-similar patterns, as analyzed in \cite[Chapter IV]{S4}. Thus, classifying all the possible forms of self-similar solutions to Eq. \eqref{eq1} with respect to their profiles is an interesting problem, as they are expected to become patterns for general blow-up. For $\sigma=0$ and $p>m$, it is established in \cite[Theorems 3 and 4, Chapter IV.1.4]{S4} that for any $p\in(m,p_s(0))$, there exists at least one radially symmetric self-similar profile to Eq. \eqref{eq1}. Another interesting question is the multiplicity of solutions suggested in \cite[Chapter IV.1.4, p. 195]{S4} as a remark based on a linearization technique, but only in dimension $N=1$, leaving implicitly open if this multiplicity remains true in dimensions $N\geq2$. We answer affirmatively this question (at least for $p$ sufficiently close to $m$) in the current paper, both for $\sigma=0$ and for some positive values of $\sigma$. Let us stress here that we faced a \emph{serious difficulty for $\sigma>0$}: since techniques based on linearization with respect to an explicit solution (as for $\sigma=0$) are no longer possible, we had to perform a careful analysis with respect to a surface built over a non-explicit solution in order to ``mimic" this idea with $\sigma>0$.

For now, let us mention that in later works, the existence of radially symmetric, self-similar blow-up profiles has been extended up to higher critical exponents than $p_s(0)$, such as the Joseph-Lundgren critical exponent $p_{JL}$ and the Lepin critical exponent $p_L$, introduced originally for the heat equation in \cite{Le90} and then extended for $m>1$ in \cite{GV97}. Recently, the authors extended in \cite{IS24} the range of existence of self-similar solutions for $\sigma=0$ to any $p>m$, overpassing thus the Lepin critical exponent. Since the current paper only deals with the subcritical range (limited from above by $p_s(\sigma)$), we refrain from giving the explicit forms of these exponents. For $m=1$ Filippas and Tertikas \cite{FT00} classified radially symmetric self-similar solutions by showing that, in the interval $m<p<p_s(\sigma)$, there exists at least one if $\sigma\in(-2,0)$ and none if $\sigma\geq0$. In recent years, the authors started a larger project of classifying self-similar patterns to Eq. \eqref{eq1} in different ranges limited by its exponents and established new and rather unexpected types of behavior of them, with a significant influence of the weight $|x|^{\sigma}$, such as, for example, in \cite{IS21, ILS23, IMS23} (for $p\in(1,m)$) or \cite{IS20, IS22, IL22} (for $p=m$). Let us thus introduce and describe self-similar solutions to Eq. \eqref{eq1} in the next paragraph.

\medskip

\noindent \textbf{Radially symmetric self-similar solutions.} Our main objects of study are the radially symmetric self-similar solutions to Eq. \eqref{eq1} presenting finite time blow-up. They have the general form
\begin{equation}\label{SSS}
u(x,t)=(T-t)^{-\alpha}f(\xi), \qquad \xi=|x|(T-t)^{-\beta}.
\end{equation}
Introducing the ansatz \eqref{SSS} into Eq. \eqref{eq1}, we deduce explicit expressions for the self-similar exponents
\begin{equation}\label{SSexp}
\alpha=\frac{\sigma+2}{L}, \qquad \beta=\frac{p-m}{L}, \qquad L=\sigma(m-1)+2(p-1)>0,
\end{equation}
together with the differential equation solved by the profiles
\begin{equation}\label{SSODE}
(f^m)''(\xi)+\frac{N-1}{\xi}(f^m)'(\xi)-\alpha f(\xi)-\beta\xi f'(\xi)+\xi^{\sigma}f(\xi)^p=0.
\end{equation}
We are thus interested in classifying the profiles $f(\xi)$ solutions to Eq. \eqref{SSODE} according to their behavior at both ends. Let us stress here that, in the subcritical range $p\in(m,p_s(\sigma))$, this task has been performed in \cite[Chapter IV.1.4]{S4} for the homogeneous case $\sigma=0$, and the conclusion has been the existence of at least a self-similar strictly positive and monotone (decreasing) profile, which satisfies the following local behavior at the origin
\begin{equation}\label{beh.P0}
f(0)=A>0, \qquad f'(0)=0
\end{equation}
as well as the decay rate
\begin{equation}\label{beh.Q1.hom}
f(\xi)\sim C\xi^{-2/(p-m)}, \qquad {\rm as} \ \xi\to\infty, \qquad C>0.
\end{equation}
The multiplicity of such self-similar profiles (if abandoning the monotonicity condition) is established in dimension $N=1$ and implicitly left open in dimension $N\geq2$. More recently, the research of a Bulgarian team \cite{DDV1, DDV2} have given numerical evidence, in any space dimension, that as $p$ tends to $m$ from the right, the formation of more and more patterns $f(\xi)$ with an increasing number of relative maxima and minima over $(0,\infty)$ occurs. One important goal of our work is to understand why this is true and give a proof to it, showing thus that multiplicity of profiles holds true in any space dimension. To fix the elements we work with, we list below the local behaviors of the profiles $f(\xi)$ that we look for. Apart from \eqref{beh.P0}, the following two initial behaviors are possible:
\begin{equation}\label{beh.P3}
f(0)=0, \qquad f(\xi)\sim\left[\frac{m-1}{2m(mN-N+2)}\right]^{1/(m-1)}\xi^{2/(m-1)}, \qquad {\rm as} \ \xi\to0,
\end{equation}
and profiles with a ``dead-core" on some interval $(0,\xi_0)$, $\xi_0\in(0,\infty)$, that is,
\begin{equation}\label{beh.Q5}
f(\xi)=0, \ \xi\in[0,\xi_0], \qquad f(\xi)>0, \ \xi\in(\xi_0,\xi_0+\delta), \qquad (f^m)'(\xi_0)=0,
\end{equation}
for some $\delta>0$. Both profiles satisfying \eqref{beh.P3} and \eqref{beh.Q5} give rise through \eqref{SSS} to self-similar weak solutions to Eq. \eqref{eq1}, since they satisfy the contact condition $(f^m)'=0$ at the ``interface" point (that is, either at $\xi=0$ in the case of \eqref{beh.P3}, or at $\xi=\xi_0$ in the case of \eqref{beh.Q5}), according to \cite[Section 9.8]{VPME}. With respect to the local behavior as $\xi\to\infty$, the only interesting one is a generalization of \eqref{beh.Q1.hom} depending on $\sigma$, that is,
\begin{equation}\label{beh.Q1}
f(\xi)\sim C\xi^{-(\sigma+2)/(p-m)}, \qquad {\rm as} \ \xi\to\infty, \qquad C>0.
\end{equation}
After these preparations, we are in a position to state the main results of this work.

\medskip

\noindent \textbf{Main results.} We begin with what we consider the most interesting result, which gives an answer to the question of multiplicity of self-similar profiles in dimension $N\geq2$ suggested in \cite[Theorems 3 and 4, Chapter IV.1.4]{S4} and the remark therein, as well as an extension of it for $\sigma>0$ but sufficiently close to zero.
\begin{theorem}[Existence and multiplicity, $\sigma\geq0$ small]\label{th.mult}
Let $N\geq2$. There exists $\sigma_0\in(0,N(m-1)/(m+1))$ such that, for any $\sigma\in[0,\sigma_0)$, there exists a decreasing sequence
\begin{equation}\label{sequence.p}
p_s(\sigma)\geq p_1(\sigma)\geq p_2(\sigma)\geq...\geq p_k(\sigma)\geq...>m, \qquad \lim\limits_{k\to\infty}p_k(\sigma)=m,
\end{equation}
such that, for any $p\in(p_{i+1}(\sigma),p_{i}(\sigma))$, $i=1,2,...$, there are:
\begin{itemize}
\item at least $i$ different self-similar solutions to Eq. \eqref{eq1} in the form \eqref{SSS} whose profiles satisfy the local behavior \eqref{beh.P0} at $\xi=0$, with either zero, one,..., $i-1$ points of positive relative minima. In particular, the first profile of this series is decreasing for $\sigma=0$ and has a single maximum point for $\sigma\in(0,\sigma_0)$.
\item at least $i$ different self-similar solutions to Eq. \eqref{eq1} in the form \eqref{SSS} whose profiles satisfy the dead-core local behavior \eqref{beh.Q5}, with either zero, one,..., $i-1$ points of positive relative minima. In particular, the first profile of this series has a single maximum point for $\sigma\in[0,\sigma_0)$.
\item \textbf{improvement for $\sigma=0$:} in the case $\sigma=0$, for the profiles with local behavior \eqref{beh.P0} as $\xi\to0$, we can let
$$
p_1(0)=p_s(0), \qquad p_k(0)=\min\left\{\frac{mk-1}{k-1},p_s(0)\right\}, \qquad {\rm for} \ k\geq2.
$$
In particular, we have at least $k$ different solutions for any $p\in(m,p_s(0))$, provided $k$ is the closest integer smaller than
\begin{equation}\label{KMN}
K(m,N):=\frac{mN-N+2m+2}{4m}.
\end{equation}
\end{itemize}
All these profiles present the decay rate as $\xi\to\infty$ given by \eqref{beh.Q1}. Moreover, there exists at least one self-similar profile with the local behavior given by \eqref{beh.P3} as $\xi\to0$ and \eqref{beh.Q1} as $\xi\to\infty$.
\end{theorem}
This theorem confirms the numerical evidence given in \cite{DDV1, DDV2}, as well as giving a precise classification of these profiles according to their behavior and the number of critical points of them. We also produce some numerical experiments whose outcome can be seen in Figures \ref{fig1} and \ref{fig2} (inserted in the proof of Theorem \ref{th.mult}, in Section \ref{sec.mult}), in order to ease the reading of the proof, plotting several profiles with different numbers of maxima and minima.

Up to our knowledge, the multiplicity of self-similar blow-up solutions when $\sigma=0$ and in dimension $N\geq2$ and the existence and multiplicity of different types of profiles with $f(0)=0$ and one of the two local behaviors \eqref{beh.P3} or \eqref{beh.Q5}, as well as with $\sigma>0$, appear to be completely new. We thus think that this classification of self-similar blow-up patterns adds up strongly to the knowledge of Eq. \eqref{eq1}, while opening up a number of interesting questions related to their stability to small perturbations and to their ``basin of attraction", that is, the set of general (even non-radial) solutions to Eq. \eqref{eq1} whose behavior near blow-up (that is, as $t\to T$) is given by the self-similar solutions classified in Theorem \ref{th.mult}. We believe that these questions are rather difficult and the approach to them should be given by some bifurcation and spectral analysis, employing techniques that might be inspired by the ones used with success in different situations when similar series of solutions have been established near a specific threshold (which in our case is $p=m$), see for example papers such as \cite{vdBV02, AvdBH03}.

\medskip

\noindent \textbf{Dead-core profiles}. Let us stress once more at this point that dead-core profiles appear to be new in this range of exponents, and their multiplicity for $\sigma\in(0,\sigma_0)$ is proven as well in Theorem \ref{th.mult}. Formation of dead-cores is a rather classical fact in diffusion equations involving a reaction, but more commonly an absorption term in bounded domains (see for example works such as \cite{BS84, GLS10, S19} and references therein), but self-similar solutions whose profiles present dead-cores have been less investigated. An unexpected feature of our dead-core profiles is their co-existence together with the more standard profiles with $f(0)>0$, $f'(0)=0$. Indeed, in the limiting case $p=m$ dead-core profiles also exist (see \cite[Theorem 1.3]{IS22}), but exactly in the range of $\sigma$ where the positive profiles at $\xi=0$ cease to exist.

\medskip

In a striking contrast with Theorem \ref{th.mult}, things change a lot when either $\sigma$ or $p$ become sufficiently large, and in fact, we are able to show in this case non-existence of radially symmetric self-similar solutions in the form \eqref{SSS}. We state below the precise result.
\begin{theorem}[Non-existence, $\sigma>0$ larger]\label{th.nonexist}
Let $\sigma>0$, $N\geq1$ and $p\in(m,p_s(\sigma))$. If
\begin{equation}\label{sigmastar}
\sigma\geq\sigma^*:=\frac{mN+2}{m-1},
\end{equation}
then Eq. \eqref{eq1} does not admit any radially symmetric self-similar solution in the form \eqref{SSS}.
\end{theorem}
We remind here that, for $\sigma=0$, existence of self-similar solutions in the form \eqref{SSS} is ensured by \cite[Theorem 4, Chapter IV.1.4]{S4}. We thus see that, if we take $\sigma>0$ and let it increase, the range of existence of such solutions decreases and even vanishes at all starting from a sufficiently large $\sigma$ (depending only on $m$ and $N$). This stems from a similar limitation for existence which occurs for $p=m$, as established (with sharp estimates on $\sigma$) in \cite{IL22, IS22}. The proof will combine arguments based on a Pohozaev identity with other ones based on the geometry of a phase space associated to Eq. \eqref{eq1}. An interesting question that springs to mind after a closer reading of the proof of Theorem \ref{th.nonexist} is related to the optimality of the exponent $\sigma^*$ for non-existence. Indeed, except for what it appears to us to be a technical limitation, one would have the impression that $\sigma^*$ could be replaced by
$$
\sigma_*=\frac{N(m-1)}{m+1}<\sigma^*,
$$
which would also match perfectly the limiting case $p=m$, where an analogous non-existence result has been established in \cite{IL22}. Recall also here that, in dimension $N\geq4$, the precise exponent limiting the existence and non-existence range when $p=m$ has been established in \cite{IS22,IL22}, namely
$$
\sigma_c=\frac{2(N-1)(m-1)}{3m+1}.
$$

\medskip

We finally turn our attention towards $\sigma<0$, when Eq. \eqref{eq1} becomes a Hardy-type equation, involving a singular potential. Beginning with the seminal work by Baras and Goldstein \cite{BG84}, equations involving singular potentials came under strong consideration and, at least for $m=1$, there is an increasing amount of interesting recent results, all published in the last decade, dealing with qualitative properties of the solutions, see for example \cite{BSTW17, BS19, CIT21a, CIT21b, T20, HT21} and references therein. Concerning singular potentials for Eq. \eqref{eq1} with $m>1$, the authors and their collaborators addressed the question of self-similar solutions with either finite time blow-up or global existence for $t\in(0,\infty)$ in the range $1<p<m$ and either $\sigma\in(-2,0)$ in \cite{IMS23, ILS23} or $\sigma=-2$ in \cite{IS23c}. While the range $p=m$ is subject to an ongoing work, we give an affirmative answer to the question of existence of blow-up self-similar solutions \eqref{SSS} in the range $\sigma\in(-2,0)$ and $p\in(m,p_s(\sigma))$ in the next theorem. As a preliminary for the statement of it, we give a more precise local behavior of the profiles $f(\xi)$ as $\xi\to0$, namely
\begin{equation}\label{beh.P0.neg}
f(\xi)\sim\left[K+\frac{p-m}{m(N+\sigma)(\sigma+2)}\xi^{\sigma+2}\right]^{-1/(p-m)}, \qquad {\rm as} \ \xi\to0.
\end{equation}
Notice that, if $\sigma\in(-1,0)$, \eqref{beh.P0.neg} implies the properties \eqref{beh.P0}, thus solutions in the form \eqref{SSS} with profiles $f$ as in \eqref{beh.P0.neg} are true weak solutions to Eq. \eqref{eq1}. If $\sigma\in(-2,-1]$ and $N\geq2$, we can observe that the slope condition at the origin $f'(0)=0$ is no longer fulfilled by profiles as in \eqref{beh.P0.neg}, and thus we are dealing with supersolutions at $x=0$. However, the branch of supersolutions we will find is a natural continuation in the range $-2<\sigma\leq-1$ of the weak solutions in the form \eqref{SSS} found for $\sigma\in(-1,0)$, hence we will make an abuse of language to keep calling them ``self-similar solutions with a peak" also in the range $\sigma\in(-2,-1]$, a convention which has been employed already in literature when dealing with singular potentials, see for example \cite{RV06}. With this in mind, we are in a position to state our existence theorem for $\sigma\in(-2,0)$.
\begin{theorem}[Existence for $\sigma\in(-2,0)$]\label{th.negative}
Let $\sigma\in(-2,0)$ such that $\sigma>-N$. Then, for any $p\in(m,p_s(\sigma))$, there exists at least one self-similar solution (with the ``abuse of language" previously explained for $\sigma\in(-2,-1]$) to Eq. \eqref{eq1} in the form \eqref{SSS}. Its profile has the local behaviors \eqref{beh.P0.neg} as $\xi\to0$ and \eqref{beh.Q1} as $\xi\to\infty$.
\end{theorem}
This result extends to the range $m>1$ the similar existence result for the range $\sigma\in(-2,0)$ established in the semilinear case $m=1$ in \cite{FT00}.

\medskip

\noindent \textbf{A remark on blow-up sets}. All the self-similar solutions in the form \eqref{SSS} obtained in the previous theorems have the singleton $\{0\}$ as blow-up set. Indeed, for any $x\in\real^N$, $x\neq0$ fixed, we observe that, as $t\to T$,
\begin{equation*}
\begin{split}
u(x,t)=(T-t)^{-\alpha}f(|x|(T-t)^{-\beta})&\sim C(T-t)^{-\alpha+(\sigma+2)\beta/(p-m)}|x|^{-(\sigma+2)/(p-m)}\\&=C|x|^{-(\sigma+2)/(p-m)},
\end{split}
\end{equation*}
which remains bounded for any $x\neq0$. On the contrary, profiles with behavior as in \eqref{beh.P0} or \eqref{beh.P0.neg} blow up obviously at $x=0$ with a rate $f(0)(T-t)^{-\alpha}$. For the profiles with one of the local behaviors \eqref{beh.P3} or \eqref{beh.Q5}, if $\xi_0\in(0,\infty)$ is their (first) maximum point, then $u(x,t)$ given by \eqref{SSS} attains a local maximum at $r_0=|x|=(T-t)^{\beta}\xi_0$, which tends to zero as $t\to T$, and the value of $u$ at this maximum is given by
$$
u(r_0,t)=(T-t)^{-\alpha}f(\xi_0)\to\infty, \qquad {\rm as} \ t\to T.
$$
We thus find that the origin belongs to the blow-up set of $u$ according to the definition of blow-up set with limits (see \cite[Section 24]{QS}). The uniform boundedness at any $x\neq0$ as $t\to T$ then gives that the blow-up set of $u$ identifies with the singleton $\{0\}$, as claimed. The blow-up rate of these solutions as $t\to T$ is $(T-t)^{-\alpha}$, thus we are dealing with a blow-up of type I.

\medskip

\noindent \textbf{Organization of the paper}. The general technique employed in the proofs is that of analyzing a dynamical systems associated to Eq. \eqref{SSODE} through a change of variable, already used with success by the authors in recent works on the fast diffusion equation with source \cite{IMS23b, IS23b}. This change of variable, together with the analysis of its equilibrium points, will be considered in the rather technical and preliminary Sections \ref{sec.local} and \ref{sec.infty}, followed by some preparatory global results in Section \ref{sec.prep}. With the tools developed in these sections, we may then pass to the proofs of the main results. Since Theorem \ref{th.negative} has a simpler and rather direct proof, we will begin with it, and this is the subject of the shorter Section \ref{sec.negative}. It is then the turn for the main and more complex result of the present paper, Theorem \ref{th.mult}, whose proof will follow by induction employing a mix of dynamical systems techniques. This proof is performed in Section \ref{sec.mult}. Finally, the proof of the non-existence Theorem \ref{th.nonexist} will join two very different ideas: a Pohozaev-type identity covering the interval $m<p\leq p_F(\sigma)$, and the construction of suitable invariant regions in the phase space of the dynamical system, preventing the orbit to connect to the critical point ``responsible" for the local behavior \eqref{beh.Q1} as $\xi\to\infty$, in order to cover the range $p_F(\sigma)<p<p_s(\sigma)$. This will be achieved in the final Section \ref{sec.nonexist}.

\section{The dynamical system. Critical points}\label{sec.local}

Assume from now on that $N\geq3$, the special dimensions $N=1$ and $N=2$ being considered in a separate paragraph. Assume also that $f(\xi)$ is a profile of a self-similar solution to Eq. \eqref{eq1} in the form \eqref{SSS}, thus it solves the differential equation \eqref{SSODE}. Starting from $f(\xi)$, we define the following new variables:
\begin{equation}\label{PSchange}
X(\eta)=\frac{\alpha}{m}\xi^2f^{1-m}(\xi), \qquad Y(\eta)=\frac{\xi f'(\xi)}{f(\xi)}, \qquad Z(\eta)=\frac{1}{m}\xi^{\sigma+2}f^{p-m}(\xi),
\end{equation}
depending on the new independent variable $\eta=\ln\,\xi$. Straightforward calculations show that \eqref{SSODE} is mapped, in these new variables,  into the system
\begin{equation}\label{PSsyst}
\left\{\begin{array}{ll}\dot{X}=X(2-(m-1)Y), \\ \dot{Y}=X-(N-2)Y-Z-mY^2+\frac{p-m}{\sigma+2}XY, \\ \dot{Z}=Z(\sigma+2+(p-m)Y).\end{array}\right.
\end{equation}
Since we are looking only to non-negative self-similar solutions, we are thus only interested in the quadrant $X\geq0$, $Z\geq0$, while the planes $\{X=0\}$ and $\{Z=0\}$ are invariant for the system \eqref{PSsyst}. Observe thus that $Y$ is the only variable that might change sign, and according to its definition in \eqref{PSchange}, such a change of sign corresponds to a critical point of the profile $f(\xi)$. The main idea throughout this work is to analyze the trajectories of the system \eqref{PSsyst} and select those ones that lead to the desired local behaviors at both ends when undoing the change of variable \eqref{PSchange}.

The system \eqref{PSsyst} has four finite critical points
\begin{equation*}
\begin{split}
P_0=(0,0,0), \ &P_1=\left(0,-\frac{N-2}{m},0\right), \ P_2=\left(0,-\frac{\sigma+2}{p-m},\frac{(\sigma+2)[p(N-2)-m(N+\sigma)]}{(p-m)^2}\right),\\
&P_3=\left(\frac{2(\sigma+2)(mN-N+2)}{L(m-1)},\frac{2}{m-1},0\right),
\end{split}
\end{equation*}
where $L>0$ has been defined in \eqref{SSexp}. Observe that the critical point $P_2$ only exists if $N\geq3$ and $p>p_c(\sigma)$, where
\begin{equation}\label{crit.p}
p_c(\sigma):=\left\{\begin{array}{ll}\frac{m(N+\sigma)}{N-2}, & N\geq3, \\ +\infty, & N\in\{1,2\},\end{array}\right.
\end{equation}
is another important critical exponent for Eq. \eqref{eq1}, already defined for example in \cite[Chapter IV]{S4} for $\sigma=0$ and recently in \cite{IS23b, IMS23b} for general $\sigma$. Let us denote for simplicity
\begin{equation}\label{zp2}
Z(P_2):=\frac{(N-2)(\sigma+2)(p-p_c(\sigma))}{(p-m)^2}, \qquad N\geq3, \qquad X(P_3):=\frac{2(\sigma+2)(mN-N+2)}{L(m-1)}.
\end{equation}
We analyze below the local dynamics of the system \eqref{PSsyst} in a neighborhood of these points. Some details in the following technical lemmas are given in \cite{IS23b}, but we will keep the presentation as self-contained as possible.
\begin{lemma}[Local analysis near $P_0$]\label{lem.P0}
The linearization of the system \eqref{PSsyst} in a neighborhood of $P_0$ has a two-dimensional unstable manifold and a one-dimensional stable manifold. The orbits on the two-dimensional unstable manifold contain(after undoing the change of variable \eqref{PSsyst}) profiles $f(\xi)$ with either the local behavior \eqref{beh.P0} if $\sigma\geq0$, or \eqref{beh.P0.neg} if $-2<\sigma<0$, while the one-dimensional stable manifold is fully included in the invariant $Y$-axis. More precisely, if $\sigma\geq0$, the local behavior of these profiles up to the second order is given by
\begin{equation}\label{beh.P0.pos}
\begin{split}
&f(\xi)\sim\left[D+\frac{\alpha(m-1)}{2mN}\xi^2\right]^{1/(m-1)}, \qquad {\rm if} \ \sigma>0,\\
&f(\xi)\sim\left[\left(\frac{D}{p-1}\right)^{(m-1)/(p-1)}+\frac{(1-D)(m-1)}{2mN(p-1)}\xi^2\right]^{1/(m-1)}, \qquad {\rm if} \ \sigma=0,
\end{split}
\end{equation}
with $D>0$ arbitrary.
\end{lemma}
\begin{proof}
The linearization of the system \eqref{PSsyst} in a neighborhood of the critical point $P_0$ has the matrix
$$
M(P_0)=\left(
         \begin{array}{ccc}
           2 & 0 & 0 \\
           1 & -(N-2) & -1 \\
           0 & 0 & \sigma+2 \\
         \end{array}
       \right),
$$
with two positive eigenvalues $\lambda_1=2$ and $\lambda_3=\sigma+2$ and one negative eigenvalue $\lambda_2=-(N-2)$ under the assumption $N\geq3$. Since the eigenvector $e_2=(0,1,0)$ corresponding to $\lambda_2$ is contained in the $Y$-axis which is invariant for the system \eqref{PSsyst}, the uniqueness of the stable manifold \cite[Theorem 3.2.1]{GH} gives that it is contained completely in this $Y$-axis. On the contrary, the unstable manifold is tangent to the plane spanned by the two eigenvectors
\begin{equation}\label{eigen.P0}
e_1=(N,1,0), \qquad e_3=(0,1,-(N+\sigma)).
\end{equation}
We next observe that, in a first order approximation, the first and third equations of the system \eqref{PSsyst} imply that
\begin{equation}\label{interm1}
Z(\eta)\sim CX^{(\sigma+2)/2}(\eta), \qquad {\rm as} \ \eta\to-\infty, \qquad C>0,
\end{equation}
which implies in particular that either $Z(\eta)$ or $X(\eta)$ dominate as $\eta\to-\infty$, depending on $\sigma$, as follows:
\begin{equation}\label{interm3}
\lim\limits_{\eta\to-\infty}\frac{Z(\eta)}{X(\eta)}=\left\{\begin{array}{lll}0, & {\rm if} \ \sigma>0,\\ C\in(0,\infty), & {\rm if} \ \sigma=0, \\ +\infty, & {\rm if} \ \sigma<0.\end{array}\right.
\end{equation}
Substituting \eqref{interm1} into the second equation of the system \eqref{PSsyst} and neglecting the quadratic terms, we deduce by integration that
\begin{equation}\label{interm2}
Y(\eta)\sim\frac{X(\eta)}{N}-\frac{Z(\eta)}{N+\sigma}+C_1X(\eta)^{-(N-2)/2}, \qquad {\rm as} \ \eta\to-\infty.
\end{equation}
Since the orbits are supposed to go out of $P_0$, we are forced to consider $C_1=0$. Moreover, by plugging \eqref{interm3} into \eqref{interm2} and neglecting the lower order term among $X(\eta)$ or $Z(\eta)$ according to \eqref{interm3}, we are left with the following equivalences as $\eta\to-\infty$:
\begin{equation}\label{interm4bis}
Y(\eta)\sim\left\{\begin{array}{ll}\frac{X(\eta)}{N}, & {\rm if} \ \sigma>0, \\[1mm] \frac{-Z(\eta)}{N+\sigma}, & {\rm if} \ \sigma\in(-2,0), \ \sigma>-N, \\[1mm] \frac{(1-C)X(\eta)}{N}, & {\rm if} \ \sigma=0,\end{array}\right.
\end{equation}
Undoing first the change of variable \eqref{PSchange} over \eqref{interm1}, we infer that the orbits contained on the unstable manifold of $P_0$ represent profiles with $f(0)>0$. With this information in hand, by undoing \eqref{PSchange} on the behaviors given in \eqref{interm4bis}, we find the more precise local behaviors \eqref{beh.P0.neg}, respectively \eqref{beh.P0.pos}, depending on the sign of $\sigma$, as claimed.
\end{proof}
\begin{lemma}[Local analysis near $P_1$]\label{lem.P1}
The critical point $P_1$ is an unstable node if $p\in(m,p_c(\sigma))$ and has a two-dimensional unstable manifold contained in the invariant plane $\{Z=0\}$ and a one-dimensional stable manifold contained in the invariant plane $\{X=0\}$ if $p>p_c(\sigma)$. The orbits going out of it for $p\in(m,p_c(\sigma))$ contain profiles presenting a vertical asymptote
\begin{equation}\label{beh.P1}
f(\xi)\sim C\xi^{-(N-2)/m}, \qquad {\rm as} \ \xi\to0, \qquad C>0.
\end{equation}
\end{lemma}
This lemma is practically identical to \cite[Lemma 2.2]{IS23b}.
\begin{proof}
The linearization of the system \eqref{PSsyst} in a neighborhood of $P_1$ has the matrix
$$
M(P_1)=\left(
         \begin{array}{ccc}
           \frac{mN-N+2}{m} & 0 & 0 \\[1mm]
           \frac{(N-2)(p_c(\sigma)-p)}{m(\sigma+2)} & N-2 & -1 \\[1mm]
           0 & 0 & \frac{(N-2)(p_c(\sigma)-p)}{m} \\
         \end{array}
       \right),
$$
with eigenvalues
$$
\lambda_1=\frac{mN-N+2}{m}>0, \qquad \lambda_2=N-2>0, \qquad \lambda_3=\frac{(N-2)(p_c(\sigma)-p)}{m}.
$$
The sign of $\lambda_3$ depends on whether $p>p_c(\sigma)$ or $p<p_c(\sigma)$ in an obvious way: if $p<p_c(\sigma)$ then we have an unstable node, while for $p>p_c(\sigma)$ we have a saddle and it suffices to compute eigenvectors in order to see that the stable manifold is contained in the plane $\{X=0\}$ and the unstable manifold is contained in the plane $\{Z=0\}$ (reminding that these two planes are invariant for the system \eqref{PSsyst} and that stable and unstable manifolds are unique \cite[Theorem 3.2.1]{GH}). The local behavior \eqref{beh.P1} follows from undoing the change of variable \eqref{PSchange} on the limiting behavior $Y(\eta)\to-(N-2)/m$ as $\eta\to-\infty$.
\end{proof}
With respect to the critical point $P_2$, which exists only if $p>p_c(\sigma)$, we have:
\begin{lemma}[Local analysis near $P_2$]\label{lem.P2}
The critical point $P_2$ is an unstable node or focus if $p\in(p_c(\sigma),p_s(\sigma))$. The orbits going out of it contain profiles with a vertical asymptote as follows:
\begin{equation}\label{beh.P2}
f(\xi)\sim K\xi^{-(\sigma+2)/(p-m)}, \ {\rm as} \ \xi\to0, \qquad K=\left[\frac{m(\sigma+2)(N-2)(p-p_c(\sigma))}{(p-m)^2}\right]^{1/(p-m)}.
\end{equation}
\end{lemma}
\begin{proof}
The proof is straightforward by computing the matrix of the linearization and it is identical to the one of \cite[Lemma 2.3]{IS23b} where full calculation details are given. The local behavior \eqref{beh.P2} follows immediately by undoing the definition of $Z$ in \eqref{PSchange} over the limit $Z(\eta)\to Z(P_2)$ as $\eta\to-\infty$, with $Z(P_2)$ defined in \eqref{zp2}.
\end{proof}
We notice that $P_1=P_2$ for $p=p_c(\sigma)$, thus the system bifurcates at this value. However, these two points will not be too important in our further analysis. The critical point $P_3$, in change, uncovers one of the ``fundamental" local behaviors for solutions to Eq. \eqref{eq1}.
\begin{lemma}[Local analysis near $P_3$]\label{lem.P3}
The critical point $P_3$ is a saddle having a two-dimensional stable manifold fully included in the invariant plane $\{Z=0\}$ and a one-dimensional unstable manifold which consists of a single trajectory. The profiles contained in it have the local behavior \eqref{beh.P3} as $\xi\to0$.
\end{lemma}
\begin{proof}
The linearization of the system \eqref{PSsyst} in a neighborhood of $P_3$ has the matrix
$$
M(P_3)=\left(
  \begin{array}{ccc}
    0 & -\frac{2(\sigma+2)(mN-N+2)}{L} & 0 \\[1mm]
    \frac{L}{(m-1)(\sigma+2)} & A(m,N,p,\sigma) & -1 \\[1mm]
    0 & 0 & \frac{L}{m-1} \\
  \end{array}
\right),
$$
where $L>0$ is defined in \eqref{SSexp} and
$$
A(m,N,p,\sigma)=-\frac{(1-m)^2N(\sigma+2)+2(m^2-1)\sigma+4(mp-1)}{L(m-1)}.
$$
The eigenvalues of $M(P_3)$ satisfy that $\lambda_3=L/(m-1)>0$ and
$$
\lambda_1\lambda_2=\frac{2(mN-N+2)}{m-1}>0, \qquad \lambda_1+\lambda_2=A(m,N,p,\sigma)<0,
$$
hence the two eigenvalues $\lambda_1$ and $\lambda_2$ are either both real and negative, or complex conjugate numbers with negative real part. In both cases, the corresponding stable manifold, due to its uniqueness and invariance of the plane $\{Z=0\}$, is fully contained in this plane. The unstable manifold of $P_3$ corresponds to the eigenvalue $\lambda_3$ and consists of a single trajectory, according to \cite[Theorem 3.2.1]{GH}. The local behavior \eqref{beh.P3} of the profiles contained in it follows readily by undoing the change of variable \eqref{PSchange} over the limit behavior
$$
X(\eta)\to X(P_2):=\frac{2(mN-N+2)(\sigma+2)}{L(m-1)}.
$$
\end{proof}

\medskip

\noindent \textbf{Dimensions $N=1$ and $N=2$}. We are left with the analysis of the critical points in dimensions $N=1$ and $N=2$. In these cases, $p_c(\sigma)=+\infty$, hence the critical point $P_2$ does not exist. Moreover, we have a few differences with respect to the analysis of $P_0$ and $P_1$. We only state here the results for the sake of completeness, the proofs being completely identical to the very detailed ones given in, for example, \cite[Section 7]{IS23b} or \cite[Lemma 3.3 and Lemma 6.1]{IMS23}, to which we refer.
\begin{lemma}\label{lem.N2}
The critical points $P_0$ and $P_1$ coincide in dimension $N=2$. Denoting by $P_0$ the resulting point, it is a saddle-node presenting a leading three-dimensional center-unstable manifold tangent to the $Y$-axis and a non-leading two-dimensional unstable manifold. The orbits in the center-unstable manifold tangent to the $Y$-axis contain profiles with a vertical asymptote at $\xi=0$ of the form
\begin{equation}\label{beh.P0.N2}
f(\xi)=D(-\ln\,\xi)^{1/m}, \qquad {\rm as} \ \xi\to0, \qquad D>0,
\end{equation}
while the orbits in the non-leading unstable manifold contain profiles with one of the local behaviors \eqref{beh.P0.neg} or \eqref{beh.P0.pos} according to the sign of $\sigma$.
\end{lemma}
Passing now to dimension $N=1$, we observe that the critical point $P_1$ moves into the positive region $\{Y>0\}$ of the phase space. Moreover, we have to restrict our analysis to $\sigma>-1$ in this special case, as many of the following developments require the general condition $N+\sigma>0$. We then have
\begin{lemma}\label{lem.N1}
In dimension $N=1$ and with $\sigma>-1$, the critical point $P_0$ is a stable node, while the critical point $P_1$ is a saddle point having a two-dimensional unstable manifold and a one-dimensional stable manifold contained in the $Y$-axis. The profiles contained in the orbits stemming from $P_0$ have either the local behavior
\begin{equation}\label{beh.P0.N1}
f(\xi)\sim\left[A-K\xi\right]^{2/(m-1)}, \qquad {\rm as} \ \xi\to0,
\end{equation}
with $A>0$ and $K\in\real\setminus\{0\}$ arbitrary constants, or one of the local behaviors \eqref{beh.P0.neg} or \eqref{beh.P0.pos} according to the sign of $\sigma$. The profiles contained in the unstable manifold of the critical point $P_1$ have the local behavior
\begin{equation}\label{beh.P1.N1}
f(\xi)\sim K\xi^{1/m}, \qquad {\rm as} \ \xi\to0, \qquad K>0.
\end{equation}
\end{lemma}
If we consider for a moment $N\geq1$ as a real parameter in the equation \eqref{SSODE} instead of an integer dimension, we are dealing with a \emph{transcritical bifurcation} at $N=2$, in the sense of \cite{S73} (see also \cite[Section 3.4]{GH}). We notice that we can keep performing the forthcoming shooting on the unstable manifold of the critical point $P_0$, in both dimensions $N=2$ (where it is completely separated from the center manifolds tangent to the $Y$ axis) and $N=1$, where we shoot from the manifold tangent to the eigenspace spanned by the eigenvectors $e_1$ and $e_3$ in \eqref{eigen.P0}. The reader may consult \cite[Section 6]{IMS23} or \cite[Section 7]{IS23b} for all the details.

This local analysis has to be completed, in order to understand all the possible behaviors of self-similar profiles to Eq. \eqref{eq1}, by the analysis of the critical points at infinity, which is the subject of the next section.

\section{Poincar\'e hypersphere. Critical points at infinity}\label{sec.infty}

Following the theory in \cite[Section 3.10]{Pe}, in order to study the infinity of the phase space associated to the system \eqref{PSsyst} we compactify the space by constructing the Poincar\'e hypersphere as follows: let
$$
X=\frac{\overline{X}}{W}, \qquad Y=\frac{\overline{Y}}{W}, \qquad Z=\frac{\overline{Z}}{W}.
$$
The critical points at space infinity, expressed in these new variables, are thus given as the solutions of the system
\begin{equation}\label{Poincare}
\left\{\begin{array}{ll}\frac{1}{\sigma+2}\overline{X}\overline{Y}[(p-m)\overline{X}-(\sigma+2)\overline{Y}]=0,\\
(p-1)\overline{X}\overline{Z}\overline{Y}=0,\\
\frac{1}{\sigma+2}\overline{Y}\overline{Z}[p(\sigma+2)\overline{Y}-(p-m)\overline{X}]=0,\end{array}\right.
\end{equation}
together with the condition of belonging to the equator of the hypersphere, which leads to $W=0$ and the additional equation $\overline{X}^2+\overline{Y}^2+\overline{Z}^2=1$, according to \cite[Theorem 4, Section 3.10]{Pe} . Since $\overline{X}\geq0$ and $\overline{Z}\geq0$, we readily find the following critical points at infinity:
\begin{equation}\label{crit.inf}
\begin{split}
&Q_1=(1,0,0,0), \ \ Q_{2,3}=(0,\pm1,0,0), \ \ Q_4=(0,0,1,0), \ \ Q_{\gamma}=\left(\gamma,0,\sqrt{1-\gamma^2},0\right),\\
&Q_5=\left(\frac{\sigma+2}{\sqrt{(\sigma+2)^2+(p-m)^2}},\frac{p-m}{\sqrt{(\sigma+2)^2+(p-m)^2}},0,0\right),
\end{split}
\end{equation}
with $\gamma\in(0,1)$. We analyze these critical points below, referring also sometimes the reader to details and calculations given in previous works such as \cite{IMS23b, IS23b} where a similar system is studied for $m<1$. As we shall see, there is no difference with respect to dimensions $N=1$ or $N=2$, thus we assume throughout this section that $N\geq1$. With respect to the critical points $Q_1$, $Q_5$ and $Q_{\gamma}$, we employ the system based on the projection on the $X$ variable according to \cite[Theorem 5(a), Section 3.10]{Pe}, that is,
\begin{equation}\label{PSinf1}
\left\{\begin{array}{ll}\dot{x}=x[(m-1)y-2x],\\
\dot{y}=-y^2+\frac{p-m}{\sigma+2}y+x-Nxy-xz,\\
\dot{z}=z[(p-1)y+\sigma x],\end{array}\right.
\end{equation}
where the variables expressed in lowercase letters are obtained from the original ones by the following change of variable
\begin{equation}\label{change2}
x=\frac{1}{X}, \qquad y=\frac{Y}{X}, \qquad z=\frac{Z}{X},
\end{equation}
and the independent variable with respect to which derivatives are taken in \eqref{PSinf1} is defined implicitly by the differential equation
\begin{equation}\label{indep2}
\frac{d\eta_1}{d\xi}=\frac{\alpha}{m}\xi f^{1-m}(\xi)=\frac{X(\xi)}{\xi}.
\end{equation}
In this system, the critical points $Q_1$, $Q_5$ and $Q_{\gamma}$ are identified as
$$
Q_1=(0,0,0), \ \ Q_5=\left(0,\frac{p-m}{\sigma+2},0\right), \ \ Q_{\gamma}=(0,0,\kappa), \ \ {\rm with} \ \kappa=\frac{\sqrt{1-\gamma^2}}{\gamma}\in(0,\infty).
$$
\begin{lemma}\label{lem.Q1Q5}
(a) The critical point $Q_1$ has a two-dimensional center manifold on which the flow behaves as on a stable manifold, and a one-dimensional unstable manifold which is fully contained in the $y$-axis. The profiles contained in the orbits entering $Q_1$ on the center manifold present the local behavior \eqref{beh.Q1} as $\xi\to\infty$.

(b) The critical point $Q_5$ is a saddle point with a two-dimensional unstable manifold, and a one-dimensional stable manifold contained in the invariant $y$-axis of the system \eqref{PSinf1}. The orbits on the unstable manifold contain profiles presenting the ``dead-core" local behavior \eqref{beh.Q5}.
\end{lemma}
\begin{proof}
(a) The linearization of the system \eqref{PSinf1} in a neighborhood of $Q_1$ has the matrix
$$M(Q_1)=\left(
         \begin{array}{ccc}
           0 & 0 & 0 \\[1mm]
           1 & \frac{p-m}{\sigma+2} & 0 \\[1mm]
           0 & 0 & 0 \\
         \end{array}
       \right).
$$
The one-dimensional unstable manifold is contained in the $y$-axis, owing to its invariance and the eigenvector $e_2=(0,1,0)$ corresponding to the only positive eigenvalue. We analyze next the center manifolds of $Q_1$ by further introducing the change of variable
\begin{equation*}
w=\frac{p-m}{\sigma+2}y+x,
\end{equation*}
which allows us to transform the system \eqref{PSinf1} into a canonical form for the center manifold theory according to \cite{Carr} (see also \cite[Section 2.12]{Pe}). With this change of variable, \eqref{PSinf1} is mapped into
\begin{equation}\label{PSinf1.bis}
\left\{\begin{array}{ll}\dot{x}=-\frac{L}{p-m}x^2+\frac{(m-1)(\sigma+2)}{p-m}xw,\\[1mm]
\dot{w}=\frac{p-m}{\sigma+2}w-\frac{\sigma+2}{p-m}w^2-\frac{p-m}{\sigma+2}xz+\frac{(m+1)(\sigma+2)-N(p-m)}{p-m}xw-\frac{m\sigma+2p-N(p-m)}{p-m}x^2,\\[1mm]
\dot{z}=-\frac{L}{p-m}xz+\frac{(p-1)(\sigma+2)}{p-m}zw,\end{array}\right.
\end{equation}
We then look for center manifolds in the form
$$
w(x,z)=ax^2+bxz+cz^2+O(|(x,z)|^3).
$$
Replacing the above ansatz into the general equation of the center manifold and performing straightforward calculations (see for example \cite[Lemma 2.1]{IMS23} or \cite[Lemma 3.1]{IMS23b} for more details), we get the equation of the center manifold of $Q_1$ as
\begin{equation}\label{cmf}
\frac{p-m}{\sigma+2}y+x=w=\frac{(\sigma+2)[m(N+\sigma)-p(N-2)]}{(p-m)^2}x^2+xz+o(|(x,z)|^2),
\end{equation}
while the flow on the center manifold is given by the reduced system (according to \cite[Theorem 2, Section 2.4]{Carr})
\begin{equation}\label{interm4}
\left\{\begin{array}{ll}\dot{x}&=-\frac{L}{p-m}x^2+x^2O(|(x,z)|),\\[1mm]
\dot{z}&=-\frac{L}{p-m}xz+xO(|(x,z)|^2).\end{array}\right.
\end{equation}
We thus deduce that the direction of the flow on the center manifolds is in fact stable (towards the critical point). The latter, together with \cite[Theorem 3.2']{Sij}, also gives that the two-dimensional center manifold is unique. For the orbits entering $Q_1$ on the center manifold, we readily get by integration in \eqref{interm4} that
\begin{equation}\label{interm6}
\frac{z(\eta_1)}{x(\eta_1)}\to k>0, \qquad {\rm as} \ \eta_1\to\infty.
\end{equation}
On the one hand, we infer from the first equation in \eqref{interm4} and \eqref{change2} that, in a neighborhood of $Q_1$ and on the center manifold, we have
\begin{equation}\label{interm7}
x(\eta_1)\sim\frac{p-m}{L\eta_1}, \qquad X(\eta_1)\sim\frac{L}{p-m}\eta_1, \qquad {\rm as} \ \eta_1\to\infty.
\end{equation}
By reverting the change of variable \eqref{indep2} and replacing \eqref{interm7} into its outcome, we further get
\begin{equation}\label{interm8}
\xi(\eta_1)=\int_0^{\eta_1}\frac{s}{X(s)}\,ds\to\int_{0}^{\infty}\frac{\eta_1}{X(\eta_1)}\,d\eta_1=+\infty,
\end{equation}
thus $\eta_1\to\infty$ on orbits entering $Q_1$ on its center manifold translates into $\xi\to\infty$ in profiles. On the other hand, from \eqref{interm6} we readily find the local behavior \eqref{beh.Q1}, by undoing \eqref{change2} and then \eqref{PSchange}, once we already know that this behavior corresponds to the limit $\xi\to\infty$ as proved above. Moreover, we also infer from \eqref{cmf} that
$$
\frac{y(\eta_1)}{x(\eta_1)}\to-\frac{\sigma+2}{p-m},
$$
whence, by undoing \eqref{change2} and then \eqref{PSchange}, we also get
$$
Y(\xi)=\frac{\xi f'(\xi)}{f(\xi)}\to-\frac{\sigma+2}{p-m}, \qquad {\rm as} \ \xi\to\infty.
$$
Taking into account the already established behavior \eqref{beh.Q1}, we infer the decay of the derivative (that will be useful later)
\begin{equation}\label{beh.Q1.deriv}
f'(\xi)\sim C\xi^{-(\sigma+2)/(p-m)-1}, \qquad {\rm as} \ \xi\to\infty.
\end{equation}

\medskip

(b) The linearization of the system \eqref{PSinf1} in a neighborhood of the critical point $Q_5$ has the matrix
$$
M(Q_5)=\left(
         \begin{array}{ccc}
           \frac{(m-1)(p-m)}{\sigma+2} & 0 & 0 \\[1mm]
           1-\frac{N(p-m)}{\sigma+2} & -\frac{p-m}{\sigma+2} & 0 \\[1mm]
           0 & 0 & \frac{(p-m)(p-1)}{\sigma+2} \\
         \end{array}
       \right),
$$
and we once more notice that the one-dimensional stable manifold is contained in the $y$-axis, by similar arguments as in part (a) of this proof. With respect to the two-dimensional unstable manifold, it is tangent to the plane spanned by the eigenvectors
$$
e_1=\left(1,\frac{\sigma+2-N(p-m)}{p-m},0\right), \qquad e_3=(0,0,1).
$$
Regarding the local behavior, we know now that $y(\eta_1)\to(p-m)/(\sigma+2)$ as $\eta_1\to-\infty$, and by replacing this in the first equation of \eqref{PSinf1} and integrating its outcome, we find by proceeding similarly as in part (a) that
$$
X(\eta_1)\sim Ce^{-(m-1)(p-m)\eta_1/(\sigma+2)}, \qquad {\rm as} \ \eta_1\to-\infty.
$$
We thus deduce, by working as in \eqref{interm8}, that, in a neighborhood of $Q_5$
$$
\xi(\eta_1)\to\int_0^{-\infty}\frac{\eta_1}{X(\eta_1)}\,d\eta_1\sim-\int_{-\infty}^{0}C\eta_1e^{(m-1)(p-m)\eta_1/(\sigma+2)}\,d\eta_1<\infty
$$
whence $\xi(\eta_1)\to\xi_0\in(0,\infty)$ from the right, as $\eta_1\to-\infty$. Moreover, with this in mind, the limit
$$
y(\eta_1)=\frac{Y(\eta_1)}{X(\eta_1)}\to\frac{p-m}{\sigma+2}, \qquad {\rm as} \ \eta_1\to-\infty,
$$
is equivalent, after undoing \eqref{PSchange}, to the local behavior
\begin{equation}\label{beh.Q5.prec}
f(\xi)\sim\left[\frac{\beta(m-1)}{2m}(\xi^2-\xi_0^2)\right]_{+}^{1/(m-1)}, \qquad {\rm as} \ \xi\to\xi_0\in(0,\infty), \ \xi>\xi_0,
\end{equation}
which satisfies the conditions in \eqref{beh.Q5}.
\end{proof}
With respect to the critical points $Q_{\gamma}$ which are identified with $(0,0,\kappa)$ in the variables $(x,y,z)$ given by \eqref{change2}, where
\begin{equation}\label{kappa.eq}
\kappa=\kappa(\gamma):=\frac{\sqrt{1-\gamma^2}}{\gamma}\in(0,\infty),
\end{equation}
we have the following:
\begin{lemma}\label{lem.Qg}
For
\begin{equation}\label{eq.gamma}
\gamma=\gamma_0:=\frac{\alpha(p-1)}{\sqrt{1+\alpha^2(p-1)^2}},
\end{equation}
the critical point $Q_{\gamma_0}$ admits a unique orbit entering it and coming from the finite part of the phase space associated to the system \eqref{PSsyst}. The profiles contained in this unique orbit have a local behavior given by
\begin{equation}\label{beh.Qg}
f(\xi)\sim\left(\frac{1}{p-1}\right)^{1/(p-1)}\xi^{-\sigma/(p-1)}, \qquad {\rm as} \ \xi\to\infty,
\end{equation}
which is either in form of a tail or it is unbounded, depending on the sign of $\sigma\in(-2,\infty)$. For $\gamma\in(0,1)$ with $\gamma\neq\gamma_0$, all the stable, unstable or center manifolds in a neighborhood of $Q_{\gamma}$ are fully contained in the plane $\{x=0\}$ and do not contain interesting trajectories for our analysis.
\end{lemma}

\noindent \textbf{Remark.} For $\sigma=0$, the orbit entering $Q_{\gamma_0}$ is fully contained in the plane $\{Y=0\}$ and corresponds to the constant profile
\begin{equation}\label{const.sol}
f(\xi)=\left(\frac{1}{p-1}\right)^{1/(p-1)}.
\end{equation}

\begin{proof}
Working with the system \eqref{PSinf1} and analyzing in this system the critical points $(0,0,\kappa)$ with $\kappa\in(0,\infty)$ as in \eqref{kappa.eq}, we translate the point $(0,0,\kappa)$ to the origin by letting $z=\overline{z}+\kappa$ in \eqref{PSinf1} to find
\begin{equation}\label{interm8bis}
\left\{\begin{array}{ll}\dot{x}=x[(m-1)y-2x],\\
\dot{y}=-y^2+\frac{p-m}{\sigma+2}y+(1-\kappa)x-Nxy-x\overline{z},\\
\dot{\overline{z}}=\overline{z}[(p-1)y+\sigma x]+\sigma\kappa x+(p-1)\kappa y,\end{array}\right.
\end{equation}
noticing that the linearization near the origin of the system \eqref{interm8bis} has a two-dimensional center manifold and an unstable manifold contained in the invariant plane $\{x=0\}$ (a fact that follows from the invariance of $\{x=0\}$, the uniqueness of the unstable manifold as given by \cite[Theorem 3.2.1]{GH} and the fact that the eigenspace of the eigenvalue $(p-m)/(\sigma+2)$ is contained in the plane $\{x=0\}$). In order to study the center manifold, we proceed as in the proof of \cite[Lemma 2.4]{ILS23} and perform the double change of variables
\begin{equation}\label{interm9}
w=\frac{p-m}{\sigma+2}y+(1-\kappa)x, \qquad v=\overline{z}-ly, \qquad l=\frac{(p-1)(\sigma+2)\kappa}{p-m}.
\end{equation}
Replacing in \eqref{interm8bis} the variables $(x,y,\overline{z})$ by the new variables $(x,w,v)$, we obtain the system
\begin{equation}\label{interm10}
\left\{\begin{array}{ll}\dot{x}&=\left[\frac{(m-1)(\sigma+2)(1-\kappa)}{p-m}-2\right]x^2+\frac{(m-1)(\sigma+2)}{p-m}wx,\\
\dot{w}&=\frac{p-m}{\sigma+2}w-\frac{\sigma+2}{p-m}w^2-\frac{p-m}{\sigma+2}vx-D_1x^2-D_2xw,\\
\dot{v}&=(\kappa\sigma+l(1-\kappa))x+D_3xv-D_4x^2-\frac{(\sigma+2)^2lp}{(p-m)^2}w^2\\&-D_5xw+\frac{\alpha(p-1)}{\beta}vw,\end{array}\right.
\end{equation}
with coefficients
\begin{equation}\label{interm10b}
\begin{split}
&D_1=\frac{(1-\kappa)[(N-2)\beta+m\alpha(1-\kappa)-\alpha\kappa(p-1)]}{\beta}, \\ &D_2=\frac{N\beta+\alpha(m+1)-\alpha\kappa(m+p)}{\beta},\\
&D_3=\sigma-\kappa+\frac{(p-1)\alpha(1-\kappa)}{\beta}, \\ &D_4=\frac{(1-\kappa)\alpha^2\kappa(p-1)[(N+\sigma)\beta+\alpha(\kappa+p-2p\kappa)]}{\beta^3},\\
&D_5=\frac{\alpha^2\kappa(p-1)[(N+\sigma)\beta+\alpha(\kappa+2p-3p\kappa)]}{\beta^3}.
\end{split}
\end{equation}
We next proceed as in the proof of \cite[Lemma 2.4]{ILS23} (to which we refer the reader for the calculation details) with the analysis of the center manifolds of the points $Q_{\gamma}$ identified as origin in the system \eqref{interm10}. In particular, following the local center manifold theory (see for example \cite[Section 2]{Carr}), we conclude that the center manifold has the form
$$
w=C_1x^2+C_2xv+C_3v^2,
$$
where the coefficients $C_1$, $C_2$ and $C_3$ can be explicitly calculated, and the flow on it is given by the reduced system
\begin{equation}\label{flowgamma}
\left\{\begin{array}{ll}\dot{x}&=\left[\frac{(m-1)(\sigma+2)(1-\kappa)}{p-m}-2\right]x^2+x^2O(|(x,v)|),\\
\dot{v}&=\left[\sigma\kappa+l(1-\kappa)\right]x+\left[\sigma-l+\frac{l(1-\kappa)}{\kappa}\right]xv\\&-Dx^2+xO(|(x,v)|^2),\end{array}\right.
\end{equation}
where
$$
D=\frac{l(1-\kappa)\alpha[(N+\sigma-l)\beta+\alpha p(1-\kappa)]}{\beta^2}, \qquad l=\frac{(p-1)(\sigma+2)\kappa}{p-m}.
$$
By performing the (implicit) change of independent variable $d\theta=xd\eta_1$ in \eqref{flowgamma} (which is equivalent to a simplification by $x$ in the right hand side of the system, which can be performed outside the plane $\{x=0\}$) and integrating the remaining system in a first approximation, it readily follows that the condition
\begin{equation}\label{interm11}
\sigma\kappa+l(1-\kappa)=0, \qquad {\rm that \ is} \qquad \kappa=\frac{1}{\alpha(p-1)},
\end{equation}
is necessary in order to exist orbits connecting to $Q_{\gamma}$ and not contained in the invariant plane $\{x=0\}$ (that is, to be still seen after the change of variable $d\theta=xd\eta_1$). We obtain the value of $\gamma_0$ from \eqref{interm11} and \eqref{kappa.eq}. For $\gamma\neq\gamma_0$, there are no orbits on the center manifold of $Q_{\gamma}$ which are not contained in the plane $\{x=0\}$. On the contrary, if $\gamma=\gamma_0$, the system \eqref{flowgamma} (taking derivatives with respect to the new variable $\theta$ defined implicitly by $d\theta=xd\eta_1$) becomes
\begin{equation}\label{flowgamma2}
\left\{\begin{array}{ll}\dot{x}=-\frac{L}{p-1}x+xO(|(x,x)|),\\
\dot{v}=\frac{L}{p-m}v-Dx+O(|(x,v)|^2),\end{array}\right.
\end{equation}
and we notice that $(x,v)=(0,0)$ is a saddle point for the system \eqref{flowgamma2}, hence by the theory in \cite[Section 3]{Sij} there exists a unique orbit entering $Q_{\gamma_0}$ on its stable manifold, coming from the region $\{x>0\}$. The same argument as in \eqref{interm7} and \eqref{interm8} (not depending on the $z$ variable) prove that along this orbit, $\eta_1\to+\infty$ implies $\xi\to\infty$, while the limit $z(\eta_1)\to\gamma_0$ as $\eta_1\to\infty$ implies, by undoing first \eqref{change2} and then \eqref{PSchange}, that the profiles $f(\xi)$ contained in this orbit present the local behavior \eqref{beh.Qg} as $\xi\to\infty$. For more calculation details, the reader is referred to the proof of \cite[Lemma 2.4]{ILS23} (see also \cite[Lemma 2.3]{IMS23}).
\end{proof}
In order to analyze locally the flow of the system \eqref{PSsyst} in a neighborhood of the critical points $Q_2$ and $Q_3$, we project on the $Y$- variable according to \cite[Theorem 5(b), Section 3.10]{Pe} and arrive to the system
\begin{equation}\label{PSinf2}
\left\{\begin{array}{ll}\pm\dot{x}=-x-Nxw+\frac{p-m}{\sigma+2}x^2+x^2w-xzw,\\[1mm]
\pm\dot{z}=-pz-(N+\sigma)zw+\frac{p-m}{\sigma+2}xz+xzw-z^2w,\\[1mm]
\pm\dot{w}=-mw-(N-2)w^2+\frac{p-m}{\sigma+2}xw+xw^2-zw^2,\end{array}\right.
\end{equation}
where the new variables $x$, $z$, $w$ are obtained from the original variables of the system \eqref{PSsyst} by the following change of variables
\begin{equation}\label{change3}
x=\frac{X}{Y}, \qquad z=\frac{Z}{Y}, \qquad w=\frac{1}{Y},
\end{equation}
and the signs plus and minus correspond to the direction of the flow in a neighborhood of the points, that is, the minus sign applies to $Q_2$ and the plus sign applies to $Q_3$. The system \eqref{PSinf2} is taken with respect to the independent variable $\eta_2$ defined implicitly by
$$
\frac{d\eta_2}{d\xi}=\frac{Y(\xi)}{\xi},
$$
and in this system (taken with the corresponding signs as explained), the critical points $Q_2$, respectively $Q_3$, are mapped into the origin of coordinates.
\begin{lemma}\label{lem.Q23}
The critical points $Q_2$ and $Q_3$ are, respectively, an unstable node and a stable node. The orbits going out of $Q_2$ correspond to profiles $f(\xi)$ such that there exists $\xi_0\in(0,\infty)$ and $\delta>0$ for which
\begin{equation}\label{beh.Q2}
f(\xi_0)=0, \qquad (f^m)'(\xi_0)=C>0, \qquad f>0 \ {\rm on} \ (\xi_0,\xi_0+\delta),
\end{equation}
while the orbits entering the stable node $Q_3$ correspond to profiles $f(\xi)$ such that there exists $\xi_0\in(0,\infty)$ and $\delta\in(0,\xi_0)$ for which
\begin{equation}\label{beh.Q3}
f(\xi_0)=0, \qquad (f^m)'(\xi_0)=-C<0, \qquad f>0 \ {\rm on} \ (\xi_0-\delta,\xi_0).
\end{equation}
\end{lemma}
\begin{proof}
It is immediate to notice that the critical point $(0,0,0)$ is an stable node in the system \eqref{PSinf2} taken with plus signs (which corresponds to $Q_3$ according to the theory in \cite[Theorem 5(b), Section 3.10]{Pe}) and an unstable node in the system \eqref{PSinf2} taken with minus signs (which corresponds to $Q_2$ according to the theory in \cite[Theorem 5(b), Section 3.10]{Pe}). With respect to the local behavior, we first show that $\eta_2\to\pm\infty$ implies $\xi\to\xi_0\in(0,\infty)$. We perform this analysis for the stable node $Q_3$ (taking the plus signs in the left hand side of the system \eqref{PSinf2}), the analysis for $Q_2$ being completely similar. Noticing from the last equation of the system \eqref{PSinf2} that, at first order of approximation and in a neighborhood of the origin,
$$
\dot{w}(\eta_2)\sim-mw(\eta_2), \qquad {\rm that \ is}, \qquad Y(\eta_2)=\frac{1}{w(\eta_2)}\sim Ce^{m\eta_2},
$$
as $\eta_2\to\infty$, we infer from the definition of the independent variable $\eta_2$ with respect to $\xi$ that
$$
\xi(\eta_2)=\int_0^{\eta_2}\frac{s}{Y(s)}\,ds\to\int_0^{\infty}\frac{\eta_2}{Y(\eta_2)}\,d\eta_2<\infty,
$$
as $\eta_2\to\infty$, since the improper integral of $C\eta_2e^{-m\eta_2}$ converges as $\eta_2\to\infty$. We thus have that, on the orbits entering the stable node $Q_3$, the profiles are supported on $[0,\xi_0]$ for some $\xi_0\in(0,\infty)$ and the local behavior near the point $Q_3$ translates into the one as $\xi\to\xi_0$, $\xi<\xi_0$. Then, in a neighborhood of $Q_3$, we deduce from the first and third equation of the system \eqref{PSinf2} that, in a first approximation,
$$
w(\eta_2)\sim Cx^m(\eta_2), \qquad {\rm as} \ \eta_2\to\infty,
$$
which in terms of profiles translates into
\begin{equation}\label{interm12}
\frac{1}{Y(\xi)}\sim\frac{C_1X(\xi)^m}{Y(\xi)^m}, \qquad {\rm as} \ \xi\to\xi_0, \ \xi<\xi_0,
\end{equation}
with $C_1<0$ a negative constant. Then \eqref{interm12} writes in terms of profiles, after easy manipulations, as
$$
f(\xi)^{m-1}f'(\xi)\sim C_2\xi^{(m+1)/(m-1)}, \qquad {\rm as} \ \xi\to\xi_0, \ \xi<\xi_0,
$$
for another generic constant $C_2<0$ (that can be explicitly expressed in terms of $C_1$). The latter leads to the local behavior \eqref{beh.Q3} as $\xi\to\xi_0\in(0,\infty)$. The proof of the local behavior \eqref{beh.Q2} is completely similar and will be omitted here.
\end{proof}
We are left with the critical point $Q_4$. We prove the following
\begin{lemma}\label{lem.Q4}
There is no orbit entering the critical point $Q_4$ and coming from the finite part of the phase space associated to the system \eqref{PSsyst}.
\end{lemma}
\begin{proof}
Since the employment of the general theory in \cite[Theorem 5(c), Section 3.10]{Pe} leads to a system whose linearization has all the eigenvalues equal to zero, it is rather difficult to analyze the critical point $Q_4$ in this way. We instead prove this lemma by passing to profiles and working on the equation \eqref{SSODE}. Indeed, assuming for contradiction that there is an orbit entering $Q_4$ and coming from the region $\{X\in(0,\infty), Z\in(0,\infty)\}$ of the phase space, on such a trajectory we would have
$$
\lim\limits_{\eta\to\infty}Z(\eta)=\infty, \qquad \lim\limits_{\eta\to\infty}\frac{Z(\eta)}{X(\eta)}=\lim\limits_{\eta\to\infty}\frac{Z(\eta)}{Y(\eta)}=\infty,
$$
which, translated in terms of profiles, means
\begin{equation}\label{limits}
\xi^{\sigma}f(\xi)^{p-1}\to\infty, \qquad \xi^{\sigma+2}f(\xi)^{p-m}\to\infty,  \qquad \frac{f'(\xi)}{\xi^{\sigma+1}f(\xi)^{p-m+1}}\to0,
\end{equation}
either when $\xi\to\infty$ or when $\xi\to\xi_0\in(0,\infty)$. Since we are in a range where $m+p>2$ is granted, the proof is completely identical to the analysis done in great detail when $m<1$ in \cite[Case 2 and Case 3, Appendix]{IMS23b} and based on the idea that $\xi^{\sigma}f^{p-1}(\xi)$ dominates over the other terms in \eqref{SSODE}, since a closer inspection of the estimates therein shows that the restriction $m<1$ does not play any role.
\end{proof}

\section{Some preparatory lemmas on the global analysis}\label{sec.prep}

In this preparatory section, we establish first the global behavior of some specific orbits completely included in the invariant planes $\{X=0\}$, $\{Z=0\}$ in the system \eqref{PSsyst}, respectively $\{x=0\}$ in the system \eqref{PSinf1}, which will be seen as limits of the manifolds we have to study in the next sections in order to prove our main theorems. We begin with the plane $\{X=0\}$, in which the system \eqref{PSsyst} reduces to the following one:
\begin{equation}\label{PSsystX0}
\left\{\begin{array}{ll}\dot{Y}=-(N-2)Y-Z-mY^2, \\ \dot{Z}=Z(\sigma+2+(p-m)Y).\end{array}\right.
\end{equation}
\begin{lemma}\label{lem.X0}
In the dynamical system \eqref{PSsystX0}, the critical point $P_0=(0,0)$ is a saddle point in dimension $N\geq3$, a saddle-node in dimension $N=2$ and an unstable node in dimension $N=1$. The orbit going out of it on its unstable manifold (or the orbit going out of it tangent to the eigenvector $E=(-1,1+\sigma)$ in dimension $N=1$) connects to the stable node $Q_3$ at infinity.
\end{lemma}
\begin{proof}
We give here the skeleton of the proof, skipping a few technical details that are given in \cite[Section 5]{IMS23b}. The stability type of $P_0$ is just a consequence of Lemmas \ref{lem.P0}, \ref{lem.N2} and \ref{lem.N1}. Let us fix now $N\geq3$. Notice first that the unstable manifold of $P_0$ is tangent to the eigenvector $E=(-1,N+\sigma)$, hence it goes out into the negative half-plane $\{Y<0\}$ and remains forever there, since the flow of the system \eqref{PSsystX0} across the axis $\{Y=0\}$ points towards the negative half-plane. Considering the following curve which connects $P_0$ to $P_1$ and its normal direction
\begin{equation}\label{iso1}
Z=-\frac{N+\sigma}{N-2}(mY+N-2)Y, \qquad \overline{n}=\left(\frac{N+\sigma}{N-2}(2mY+N-2),1\right),
\end{equation}
we observe that the flow of the system \eqref{PSsystX0} across it is given by the sign of the expression
$$
H(Y)=\frac{(N+\sigma)(p_s(\sigma)-p)}{N-2}Y^2(mY+N-2)>0,
$$
in the range where $mY+N-2>0$, that is, the $Y$-interval between $P_1$ and $P_0$. Thus, the region
$$
\mathcal{D}:=\left\{Y<0, Z>0, Z>-\frac{N+\sigma}{N-2}Y(mY+N-2)\right\}
$$
is positively invariant, that is, an orbit entering it, will lie forever inside it. One can readily notice that, for any point in $\mathcal{D}$ with $-(N-2)/m<Y<0$, we have
\begin{equation*}
\begin{split}
\dot{Y}&=-mY^2-(N-2)Y-Z<-Y(mY+N-2)+\frac{N+\sigma}{N-2}Y(mY+N-2)\\
&=\frac{\sigma+2}{N-2}Y(mY-N+2)<0,
\end{split}
\end{equation*}
while $\dot{Y}<0$ is obvious for $Y\leq-(N-2)/m$. Hence, $Y(\eta)$ is a decreasing function on any trajectory lying (forever) inside the region $\mathcal{D}$. Since a calculation of the second order of the unstable manifold of $P_0$ gives
\begin{equation}\label{interm13}
Z=-(N+\sigma)Y-\frac{(N+\sigma)p}{N+2\sigma+2}Y^2+o(Y^2),
\end{equation}
in a neighborhood of $(Y,Z)=(0,0)$ (see \cite[Section 5]{IMS23b} for the details of this calculation), we infer that the unstable manifold of $P_0$ goes into the region $\mathcal{D}$ exactly when $1<p<p_s(\sigma)$, since \eqref{interm13} coincides with \eqref{iso1} for $p=p_s(\sigma)$ and the right hand side of \eqref{interm13} decreases with respect to $p$. Thus, the unique orbit on the unstable manifold of $P_0$ remains forever inside the region $\mathcal{D}$ and consequently, $Y(\eta)$ is decreasing with $\eta$ along it. Since $Z(\eta)$ has only one change of monotonicity at $Y=-(\sigma+2)/(p-m)$, we conclude that both $Y(\eta)$ and $Z(\eta)$ are monotone at least as $\eta\to\infty$ (oscillations are not possible). The Poincar\'e-Bendixon theory \cite[Section 3.7]{Pe} and this monotonicity give that the orbit converges to a critical point as $\eta\to\infty$, and it is easy to check that this critical point can only be $Q_3$. We refer the reader to the proof of \cite[Lemma 5.2]{IMS23b} for full details. In dimension $N\in\{1,2\}$, the proof is practically immediate by simply taking $\mathcal{D}=\{Y<0\}$ and noting that $\dot{Y}(\eta)<0$ while $Y(\eta)<0$, hence $Y$ is a decreasing function of $\eta\in\real$ on the whole trajectory.
\end{proof}
The previous lemma will be very important in the sequel, and, in fact, it is this result (and its consequences for the shooting method from a large shooting parameter, as we shall see in the forthcoming sections) that was missing in \cite[Section IV.1.4, pp. 195-196]{S4} for $\sigma=0$ and in dimensions $N\geq2$ for the proof of the multiplicity, since \cite[Lemma 4, p. 193, Section IV.1.4]{S4} is only established in dimension $N=1$.

\medskip

\noindent \textbf{Remark.} The curve \eqref{iso1} gives rise to an explicit family of stationary solutions for $p=p_s(\sigma)$, given by
\begin{equation}\label{sol.sobolev}
U_C(x)=\left[\frac{(N-2)(N+\sigma)C}{(|x|^{\sigma+2}+C)^2}\right]^{(N-2)/2m(\sigma+2)},
\end{equation}
see \cite[Section 9]{IMS23b} for its detailed deduction.

\medskip

We turn now our attention to the invariant plane $\{Z=0\}$, where the system \eqref{PSsyst} becomes
\begin{equation}\label{PSsystZ0}
\left\{\begin{array}{ll}\dot{X}=X(2-(m-1)Y), \\ \dot{Y}=X-(N-2)Y-mY^2+\frac{p-m}{\sigma+2}XY.\end{array}\right.
\end{equation}
Moreover, the same invariant plane can be regarded in variables $(x,y)$ as the plane $\{z=0\}$ in the system \eqref{PSinf1}, that is
\begin{equation}\label{PSsystz00}
\left\{\begin{array}{ll}\dot{x}=x[(m-1)y-2x],\\
\dot{y}=-y^2+\frac{p-m}{\sigma+2}y+x-Nxy,\end{array}\right.
\end{equation}
We have
\begin{lemma}\label{lem.Z0}
The critical point $P_0=(0,0)$ in the system \eqref{PSsystZ0} is a saddle point in dimension $N\geq3$, a saddle-node in dimension $N=2$ and an unstable node in dimension $N=1$. The unique orbit contained in its unstable manifold for $N\geq2$ (or the unique orbit tangent to the eigenvector $E_1=(1,1)$ in dimension $N=1$) connects to the critical point $P_3$, which is a stable point for the system \eqref{PSsystZ0}. The critical point $Q_5=(0,(p-m)/(\sigma+2))$ is a saddle in the system \eqref{PSsystz00}. The unique orbit contained in its unstable manifold connects to the critical point $P_3$.
\end{lemma}
\begin{proof}
The fact that $P_0$ is a saddle in dimension $N\geq3$, a saddle-node in dimension $N=2$ and an unstable node in dimension $N=1$, and that $P_3$ is a stable point for \eqref{PSsystZ0} follow readily from Lemmas \ref{lem.P0}, \ref{lem.N2}, \ref{lem.N1}, respectively \ref{lem.P3}. Similarly, the fact that $Q_5$ is a saddle point for the system \eqref{PSsystz00} follows from Lemma \ref{lem.Q1Q5}. Moreover, the unique orbit contained in the unstable manifold of $P_0$ goes out tangent to the eigenvector $E_1=(N,1)$, thus into the positive half-plane $\{Y>0\}$. Since the flow of the system \eqref{PSsystZ0} across the axis $\{Y=0\}$ points into the positive half-plane, we conclude that this orbit stays forever in the region $\{Y>0\}$. We show first that the system \eqref{PSsystZ0} does not have any limit cycle in the region $\{Y>0\}$. To this end, we use Dulac's Criterion by multiplying the two components of the vector field of the system by the factor $X^{a}Y^{b}$, with
$$
a=-\frac{2mN-N-2m+2}{mN-N+2}, \qquad b=-\frac{mN-N+2m+2}{mN-N+2}.
$$
We obtain that the divergence of the new vector field obtained by this multiplication is given by
$$
\frac{\partial(X^aY^b\dot{X})}{\partial X}+\frac{\partial(X^aY^b\dot{Y})}{\partial Y}=bX^{a+1}Y^{b-1}+\frac{(b+1)(\sigma+2)}{p-m}X^{a+1}Y^b<0,
$$
since obviously $b<0$ and
$$
b+1=-\frac{2m}{mN-N+2}<0.
$$
Dulac's Criterion (see for example \cite[Theorem 2, Section 3.9]{Pe}) then ensures that there is no closed orbit contained entirely in the region $\{Y>0\}$. Since the axis $\{X=0\}$ is invariant and the axis $\{Y=0\}$ has a definite direction of the flow, there is no closed orbit contained in the first quadrant, including the axis. This, together with the Poincar\'e-Bendixon's theory, implies that the unique orbit contained in the unstable manifold of $P_0$ has to end up by entering a critical point lying on the closure of the first quadrant. The classification we performed in Sections \ref{sec.local} and \ref{sec.infty} shows that the only possible critical point is $P_3$, as claimed. In fact, this shows that whatever orbit existing inside the first quadrant $\{X>0,Y>0\}$ of the system \eqref{PSsystZ0} (with any possible origin) is attracted by $P_3$, which also applies for the critical point $Q_5$, completing the proof.
\end{proof}
The next preparatory result is related to a plane that can be formally seen as $\{X=\infty\}$. More precisely, we go to the system \eqref{PSinf1} and we would like to set there $x=0$. Since this system decouples, let us first replace the variable $z$ in \eqref{PSinf1} by $w=xz\geq0$, a change of variable that has been seen to be very useful in, for example, \cite[Section 4]{IS21}. Performing first this change, and letting then $x=0$ in the resulting system, we are left with the system
\begin{equation}\label{PSsystw0}
\left\{\begin{array}{ll}\dot{y}=-y^2+\frac{p-m}{\sigma+2}y-w,\\ \dot{w}=(m+p-2)yw,\end{array}\right.
\end{equation}
where derivatives are taken with respect to the independent variable $\eta_1$ given in \eqref{indep2}, with finite critical points $Q_1'=(0,0)$ and $Q_5'=((p-m)/(\sigma+2),0)$, which can be seen as the ``restrictions'' of the true critical points $Q_1$ and $Q_5$ to the system \eqref{PSsystw0}. We have the following
\begin{lemma}\label{lem.w0}
The critical point $Q_5'$ is a saddle point in the system \eqref{PSsystw0} and the critical point $Q_1'$ is a saddle-node in the system \eqref{PSsystw0}. The unique orbit contained in the unstable manifold of $Q_5'$ connects to the critical point $Q_3$, with $y(\eta_1)$ decreasing for $\eta_1\in\real$. The orbits going out of $Q_1'$ connect all of them to the critical point $Q_3$ as well, with the function $y(\eta_1)$ having at most one maximum point along these orbits.
\end{lemma}
\begin{proof}
The matrix of the linearization of the system \eqref{PSsystw0} in a neighborhood of $Q_5'$ is
$$
M(Q_5')=\left(
          \begin{array}{cc}
            -\frac{p-m}{\sigma+2} & -1 \\[1mm]
            0 & \frac{(m+p-2)(p-m)}{\sigma+2} \\
          \end{array}
        \right),
$$
thus $Q_5'$ is a saddle point. Consider now the isocline where $\dot{y}=0$, that is,
\begin{equation}\label{iso2}
-y^2+\frac{p-m}{\sigma+2}y-w=0, \qquad {\rm with \ normal} \ \overline{n}=\left(-2y+\frac{p-m}{\sigma+2},-1\right)
\end{equation}
which lies in the region $w\geq0$ if and only if $0\leq y\leq (p-m)/(\sigma+2)$. The flow of the system \eqref{PSsystw0} across the isocline is given by the sign of the expression $-(m+p-2)yw<0$, hence it points out into the increasing direction of $w$. This means that an orbit entering the region $\dot{y}<0$ remains there forever. Observing that, in a neighborhood of $Q_5'$, the isocline has the slope
\begin{equation}\label{interm15}
w\sim-\frac{p-m}{\sigma+2}\overline{y}, \qquad \overline{y}=y-\frac{p-m}{\sigma+2}
\end{equation}
while the slope of the actual orbit contained in the unstable manifold of the point $Q_5'$ is given by the eigenvector corresponding to the second eigenvalue of $M(Q_5')$, that is,
\begin{equation}\label{interm16}
w\sim-\frac{(m+p-1)(p-m)}{\sigma+2}\overline{y},
\end{equation}
we deduce from \eqref{interm15}, \eqref{interm16} and the fact that $m+p-1>1$ that the orbit stemming from $Q_5'$ enters the region $\dot{y}<0$ (superior to the isocline \eqref{iso2}) and stays forever inside it. This proves that $y(\eta_1)$ is decreasing for $\eta_1\in\real$ along this orbit. It remains to show that this orbit crosses the axis $\{y=0\}$ at a finite point, that is, there is $\eta_1\in\real$ such that $y(\eta_1)=0$. Assume for contradiction that this is not the case, hence $y(\eta_1)>0$ for any $\eta_1\in\real$. We notice from the system \eqref{PSsystw0} that $\dot{y}(\eta_1)<0$ and $\dot{w}(\eta_1)>0$ for any $\eta_1\in\real$, so that, there exist
$$
y_0:=\lim\limits_{\eta_1\to\infty}y(\eta_1)\in\left[0,\frac{p-m}{\sigma+2}\right], \qquad w_0:=\lim\limits_{\eta_1\to\infty}w(\eta_1)\in(0,\infty].
$$
If $w_0$ is finite, then $(y_0,w_0)$ would be a finite critical point of the system \eqref{PSsystw0} with $y\geq0$, $w>0$ and there is no such point. Thus, $w_0=\infty$ and thus the orbit will form a vertical asymptote from the right as $y\to y_0$. Moreover, since $y(\eta_1)$ is monotone (hence one-to-one) along the orbit stemming from $Q_5'$, for $\eta_1\in\real$, we can apply the inverse function theorem and express this trajectory as a function $w(y)$, with derivative
$$
\frac{dw}{dy}=\frac{(m+p-2)yw(y)}{-y^2+(p-m)y/(\sigma+2)-w(y)}.
$$
Passing to the limit as $y\to y_0$ and taking into account that $w(y)\to+\infty$, we find that
$$
\lim\limits_{y\to y_0, y>y_0}w'(y)=-(m+p-2)y_0
$$
which is finite, contradicting the vertical asymptote from the right of the function $w(y)$ at $y=y_0$. We thus infer that there exists $\overline{\eta}_1\in\real$ such that $y(\overline{\eta}_1)=0$ on the orbit contained in the unstable manifold of $Q_5'$, and then $y(\eta_1)<0$ for any $\eta_1>\overline{\eta}_1$. Furthermore, the latter, together with the second equation of the system \eqref{PSsystw0}, implies that $\dot{w}(\eta_1)<0$ for any $\eta_1>\overline{\eta}_1$. This shows that the orbit contained in the unstable manifold of $Q_5'$ has a critical point as limit as $\eta_1\to\infty$, lying in the half-plane $\{y<0\}$ and with finite $w$-coordinate. There are no such finite critical points, and it is then easy to derive that
$$
\lim\limits_{\eta_1\to\infty}y(\eta_1)=-\infty, \qquad \lim\limits_{\eta_1\to\infty}\frac{w(\eta_1)}{y(\eta_1)}=0,
$$
thus the orbit will reach the critical point $Q_3$.

\medskip

We are left with the critical point $Q_1'$. The matrix of the linearization of the system \eqref{PSsystw0} in a neighborhood of $Q_1'$ is given by
$$
M(Q_1')=\left(
          \begin{array}{cc}
            \frac{p-m}{\sigma+2} & -1 \\[1mm]
            0 & 0 \\
          \end{array}
        \right),
$$
thus we have a one-dimensional unstable manifold (contained in the invariant $y$-axis) and one-dimensional center manifolds. Applying the theory in \cite[Sections 2.4 and 2.5]{Carr} in order to analyze these center manifolds, we find that the Taylor expansion of them up to the second order is given by
$$
w=\frac{p-m}{\sigma+2}y-(m+p-1)y^2+o(y^2),
$$
while the flow on these center manifolds in a neighborhood of the origin is given by
$$
\dot{w}=\frac{(m+p-2)(p-m)}{\sigma+2}y^2>0.
$$
Thus, all the center manifolds have an unstable flow, all these trajectories going out into the positive half-plane $\{y>0\}$ tangent to the line
$$
w=\frac{p-m}{\sigma+2}y.
$$
Thus, they have to enter the closed region limited by the axis $\{w=0\}$ (which is invariant) and the isocline \eqref{iso2}, that is, the region where $\dot{y}>0$. Since the flow on the isocline \eqref{iso2} points into an unique direction (towards the open, exterior region where $\dot{y}<0$), it follows that all the orbits contained in center manifolds going out of $Q_1'$ will cross the isocline \eqref{iso2} at some point (and they can do this only once, due to the direction of the flow across \eqref{iso2}) and then enter and stay forever in the region where $\dot{y}<0$. Since they are bounded from above by the unique orbit contained in the unstable manifold of $Q_5'$, they have then to cross the axis $\{y=0\}$ at some finite point and enter the half-plane $\{y<0\}$. The same considerations as in the final part of the proof related to the orbit contained in the unstable manifold of $Q_5'$ show that all these orbits will end up connecting to $Q_3$, and $y(\eta_1)$ has a unique maximum point, which is exactly at the value of $\eta_1$ for which the orbit intersects the isocline \eqref{iso2}.
\end{proof}

In order to ease the reading of the previous proof, we plot in Figure \ref{fig4} the unique trajectory of the system \eqref{PSsystw0} stemming from the critical point $Q_5'$, which is decreasing with respect to $y$, and a bunch of trajectories of the same system \eqref{PSsystw0} contained in (different) center manifolds going out of the critical point $Q_1'$ and having exactly one maximum in the $y$-coordinate.

\begin{figure}[ht!]
  \begin{center}
  \includegraphics[width=11cm,height=7.5cm]{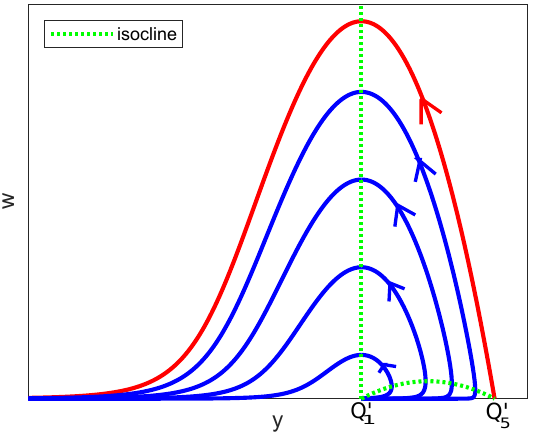}
  \end{center}
  \caption{Trajectories of the system \eqref{PSsystw0} going out of $Q_5'$ and $Q_1'$.}\label{fig4}
\end{figure}

We end this section with two more preparatory results, which are not related to any invariant plane, but to a "plane of no return". More precisely, on the one hand, we have
\begin{lemma}\label{lem.noret}
Consider the plane $\{Y=-(\sigma+2)/(p-m)\}$ in the phase space associated to the system \eqref{PSsyst}. If $m<p<p_s(\sigma)$, on any orbit contained in the unstable manifold of the critical point $P_0$, if there is $\eta_0\in\real$ such that $Y(\eta_0)\leq-(\sigma+2)/(p-m)$, then $Y(\eta)<-(\sigma+2)/(p-m)$ for any $\eta\in(\eta_0,\infty)$. Moreover, in this case the corresponding orbit connects to the stable node $Q_3$ at infinity.
\end{lemma}
\begin{proof}
We divide the proof into two steps, according to the range of $p$.

\medskip

\noindent \textbf{Step 1: $N\in\{1,2\}$ or $N\geq3$ and $m<p\leq p_c(\sigma)$}. The flow of the system \eqref{PSsyst} across the plane $\{Y=-(\sigma+2)/(p-m)\}$ (with normal vector $(0,1,0)$) is given by the sign of the expression
\begin{equation}\label{flow.plane2}
F_2(X,Z)=-\frac{(\sigma+2)[m(N+\sigma)-p(N-2)]}{(p-m)^2}-Z,
\end{equation}
which is always negative if $m(N+\sigma)-p(N-2)>0$, the latter being ensured if either $N\in\{1,2\}$ or $N\geq3$ and $p\leq p_c(\sigma)=m(N+\sigma)/(N-2)$. This shows that, in the range under analysis, if $Y(\eta_0)<-(\sigma+2)/(p-m)$ for some $\eta_0\in\real$, then $Y(\eta)<-(\sigma+2)/(p-m)$ for any $\eta\in(\eta_0,\infty)$. We next infer from the first and third equation of the system \eqref{PSsyst} that $\dot{X}(\eta)>0$ and $\dot{Z}(\eta)<0$ for any $\eta>\eta_0$. Moreover,
\begin{equation}\label{interm21}
\dot{Y}(\eta)=X(\eta)\left(1+\frac{p-m}{\sigma+2}Y(\eta)\right)-mY(\eta)\left(Y(\eta)+\frac{N-2}{m}\right)-Z(\eta)<0,
\end{equation}
for any $\eta>\eta_0$, since, for any $p\in(m,p_c(\sigma)]$, we have
$$
\frac{N-2}{m}-\frac{\sigma+2}{p-m}=\frac{p(N-2)-m(N+\sigma)}{m(p-m)}\leq0.
$$
It thus follows that all three components $X(\eta)$, $Y(\eta)$, $Z(\eta)$ are monotone along the trajectory, on $(\eta_0,\infty)$. In particular, $0<Z(\eta)<Z(\eta_0)$ for any $\eta\in(\eta_0,\infty)$, hence $Z(\eta)$ is bounded. Since there are no finite critical points in the half-space $\{Y<-(\sigma+2)/(p-m)\}$, it follows that either $X(\eta)\to\infty$ or $Y(\eta)\to-\infty$, or both, as $\eta\to\infty$. If $X(\eta)$ has a finite limit as $\eta\to\infty$, this forces $Y(\eta)\to-\infty$ as $\eta\to\infty$ and thus $X(\eta)/Y(\eta)$, $Z(\eta)/Y(\eta)$ converge to zero as $\eta\to\infty$, which means that the trajectory enters $Q_3$. On the contrary, if $X(\eta)\to\infty$ as $\eta\to\infty$, we readily infer from the first and second equation of the system \eqref{PSsyst}, the already established monotonicity of $Y(\eta)$ and the fact that $-mY^2-(N-2)Y-Z<0$ on $(\eta_0,\infty)$, that
$$
\frac{dY}{dX}<\frac{1+(p-m)/(\sigma+2)Y}{2-(m-1)Y}<\beta\left(1+\frac{p-m}{\sigma+2}Y\right),
$$
which gives, by comparison with the solutions of the actual differential equation, that, along the considered trajectory, one has
$$
1+\frac{p-m}{\sigma+2}Y(\eta)<-Ke^{\beta X(\eta)},
$$
hence $Y(\eta)/X(\eta)\to-\infty$, and thus the trajectory reaches the critical point $Q_3$, as claimed.

\medskip

\noindent \textbf{Step 2: $N\geq3$ and $p_c(\sigma)<p<p_s(\sigma)$}. Then \eqref{flow.plane2} itself does not end the argument as in the previous range, so that we also consider the cylinder with explicit equation \eqref{iso1}, but considered now for any $X\geq0$. Taking the normal direction to it as
$$
\overline{\nu}=\left(0,\frac{N+\sigma}{N-2}(2mY+N-2),1\right),
$$
the flow of the system \eqref{PSsyst} across it is given by the scalar product between $\overline{\nu}$ and its vector field, that is,
\begin{equation}\label{flow.cyl}
\begin{split}
E(X,Y;p)=\frac{N+\sigma}{N-2}&\left[(p_s(\sigma)-p)Y^2(mY+N-2)\right.\\&\left.+X\left(1+\frac{p-m}{\sigma+2}Y\right)(2mY+N-2)\right].
\end{split}
\end{equation}
In our range the first term in \eqref{flow.cyl} is clearly positive. As for the second, it can only be negative provided
\begin{equation}\label{reg.cyl}
-\frac{\sigma+2}{p-m}<Y<-\frac{N-2}{2m}.
\end{equation}
But despite this fact, any orbit lying outside the cylinder \eqref{iso1} when crossing the plane $\{Y=-(N-2)/2m\}$ will continue outside the cylinder. Indeed, on the plane $\{Y=-(N-2)/2m\}$ the cylinder \eqref{iso1} reaches its maximum amplitude with respect to the coordinate $Z$, while the $Z$ component on any trajectory of the system \eqref{PSsyst} will keep increasing until intersecting the plane $\{Y=-(\sigma+2)/(p-m)\}$, as indicated by the third equation in \eqref{PSsyst}. Thus, $Z(\eta)$ will in any case not decrease in order to reach the surface of the cylinder \eqref{iso1} in the region \eqref{reg.cyl}. Moreover, if for some $\eta\in\real$,
$$
-\frac{N-2}{m}<Y(\eta)<-\frac{\sigma+2}{p-m}, \qquad Z(\eta)>-\frac{N+\sigma}{N-2}Y(\eta)(mY(\eta)+N-2),
$$
we have for any $\sigma>-2$
\begin{equation*}
\begin{split}
\dot{Y}(\eta)&=X(\eta)\left(1+\frac{p-m}{\sigma+2}Y(\eta)\right)-Y(\eta)(mY(\eta)+N-2)-Z(\eta)\\
&<X(\eta)\left(1+\frac{p-m}{\sigma+2}Y(\eta)\right)+\frac{\sigma+2}{N-2}Y(\eta)(mY(\eta)+N-2)<0,
\end{split}
\end{equation*}
while it is completely obvious that $\dot{Y}(\eta)<0$ if $Y(\eta)\leq-(N-2)/m$. Thus, $Y(\eta)$ is a decreasing function of $\eta$ on the interval $(\eta_0,\infty)$ on any trajectory of the system \eqref{PSsyst} for which there exists $\eta_0\in\real$ such that
$$
\left\{Y(\eta_0)<-\frac{\sigma+2}{p-m}, \ Z(\eta_0)>\max\left\{0,-\frac{N+\sigma}{N-2}Y(\eta_0)(mY(\eta_0)+N-2)\right\}\right\}.
$$
Similar considerations as in the end of Step 1 prove that these orbits will reach the critical point $Q_3$. It only remains to prove that indeed, the orbits on the unstable manifold of $P_0$ go out into the exterior of the cylinder \eqref{iso1} and thus the condition on $Z(\eta_0)$ is implicit once $Y(\eta_0)<-(\sigma+2)/(p-m)$ for some $\eta_0\in\real$. The Stable Manifold Theorem \cite[Section 2.7]{Pe} states that the unstable manifold of $P_0$ is tangent to the region of the plane spanned by the eigenvectors given by \eqref{eigen.P0} and limited by the invariant planes $\{X=0\}$ and $\{Z=0\}$, that is, to vectors of the form
$$
v(\lambda):=\lambda e_1+(1-\lambda)e_3, \qquad \lambda\in(0,1).
$$
But it is obvious that $e_1$ goes into the exterior of the cylinder \eqref{iso1} (as it points out into the positive half-space $\{Y>0\}$), while the orbit tangent to the only vector $v(0)$ having no component on $e_1$ points into the exterior of the cylinder, as shown by the expansion \eqref{interm13}. It is then easy to see that any other orbit tangent to the unstable subspace of $P_0$ composed by the vectors $v(\lambda)$ points into the exterior of the cylinder \eqref{iso1} and will remain there forever, ending the proof.
\end{proof}
On the other hand, we can also characterize the trajectories of the system \eqref{PSsyst} which do not cross the plane $\{Y=-(\sigma+2)/(p-m)\}$ as follows.
\begin{lemma}\label{lem.ret}
Let $\sigma\geq0$ and $(X,Y,Z)(\eta)$ be a trajectory of the system \eqref{PSsyst} such that there exists $\eta_0\in\real$ with $X(\eta_0)>0$, $Z(\eta_0)>0$ and
\begin{equation}\label{interm22}
-\frac{\sigma+2}{p-m}<Y(\eta)<\frac{2}{m-1}, \qquad {\rm for \ any} \ \eta\in(\eta_0,\infty).
\end{equation}
Then, this trajectory enters either the critical point $Q_1$ or $Q_{\gamma_0}$ with $\gamma_0$ defined in \eqref{eq.gamma}.
\end{lemma}
\begin{proof}
The condition \eqref{interm22}, together with the first and third equation of the system \eqref{PSsyst}, imply that $X(\eta)$ and $Z(\eta)$ are increasing functions of $\eta$ on the trajectory, so that, there exist
$$
X_{\infty}:=\lim\limits_{\eta\to\infty}X(\eta)>0, \qquad Z_{\infty}:=\lim\limits_{\eta\to\infty}Z(\eta)>0.
$$
Assume for contradiction that $X_{\infty}<\infty$. Then either $Z_{\infty}<\infty$ and arguments as in the proof of \cite[Proposition 4.1]{ILS24}  based on subsequences of local maxima and minima of $Y(\eta)$ then show that the $\omega$-limit set must be a finite critical point, which is impossible as such a point does not exist in the range of \eqref{interm22} outside the invariant planes $\{X=0\}$ and $\{Z=0\}$, or $Z_{\infty}=\infty$, which also gives $Y(\eta)/Z(\eta)\to0$ and $X(\eta)/Z(\eta)\to0$ as $\eta\to\infty$ and thus our trajectory enters the critical point $Q_4$, contradicting Lemma \ref{lem.Q4}. We infer thus that $X_{\infty}=\infty$ and, consequently, $Y(\eta)/X(\eta)\to0$. Passing now to $(x,y,z)$-variables and seeing the trajectory in the system \eqref{PSinf1}, we infer that along it, components $x$ and $y$ tend to zero. We thus infer that the $\omega$-limit set of our orbit must be a segment of the critical line $\{x=0, y=0\}$ of the system \eqref{PSinf1}. Coming back to the space of profiles $(\xi,f(\xi))$, the latter conclusion is equivalent to the corresponding profile oscillating between two hyperbolas
$$
A_1\xi^{-\sigma/(p-1)}\leq f(\xi)\leq A_2\xi^{-\sigma/(p-1)}, \qquad \xi\geq\xi_0>0,
$$
for some constants $0\leq A_1<A_2<\infty$ and some $\xi_0>0$ very large. Letting then $g(\xi):=\xi^{\sigma/(p-1)}f(\xi)$, we find by direct calculation that
\begin{equation}\label{ode.g}
\begin{split}
\xi^2(g^m)''(\xi)&-\left(\frac{2m\sigma}{p-1}-N+1\right)\xi(g^m)'(\xi)+\frac{m\sigma}{p-1}\left(\frac{m\sigma}{p-1}-N+2\right)g^m(\xi)\\
&-\xi^{L/(p-1)}\left[\frac{1}{p-1}g(\xi)-\beta\xi g'(\xi)\right]+\xi^{-\sigma/(p-1)}g^p(\xi)=0.
\end{split}
\end{equation}
Let $(\xi_{k})_{k\geq1}$ be a sequence of local maxima of $g$, such that $\xi_{k}\to\infty$ and
$$
g(\xi_k)\to\limsup\limits_{\xi\to\infty}g(\xi)=L\in(0,\infty).
$$
Evaluating \eqref{ode.g} at $\xi=\xi_k$ for $k\geq1$, we obtain a contradiction, since $-\xi_{k}^{L/(p-1)}g(\xi_k)$ cannot be compensated by any other term with a plus sign in \eqref{ode.g} for $k$ sufficiently large. We reach a contradiction with the existence of an oscillating $\omega$-limit of $l_C$, as it follows that for $\xi\geq\xi_0$ sufficiently large, $g(\xi)$ has to be monotone. Hence, the $\omega$-limit set should be a critical point and Lemmas \ref{lem.Q1Q5} and \ref{lem.Qg} then establish that the only possible point is either $Q_1$ or $Q_{\gamma_0}$, as claimed.
\end{proof}

\section{The range $-2<\sigma<0$. Proof of Theorem \ref{th.negative}}\label{sec.negative}

We are now in a position to prove our main results, and we begin with the existence for $\sigma\in(-2,0)$ (or $\sigma\in(-1,0)$ if $N=1$). By gathering \eqref{interm1} and \eqref{interm2}, we can express the unstable manifold of the critical point $P_0$ as the following one-parameter family of orbits in a neighborhood of $P_0$:
\begin{equation}\label{manifP0}
(l_C): \quad Y(\eta)\sim\frac{X(\eta)}{N}-\frac{CX(\eta)^{(\sigma+2)/2}}{N+\sigma}, \qquad {\rm as} \ \eta\to-\infty,
\end{equation}
for any $C\in(0,\infty)$, extended with its limits $l_0$ (for $C=0$), which is the unique orbit contained in the invariant plane $\{Z=0\}$, and $l_{\infty}$ (corresponding to $C=\infty$, or $C^{-1}=0$ in the alternative writing by replacing $X(\eta)$ by $Z(\eta)$ according to \eqref{interm1}), which is the unique orbit contained in the invariant plane $\{X=0\}$. We next prove Theorem \ref{th.negative} by employing a shooting method with respect to the parameter $C\in(0,\infty)$.

\begin{proof}[Proof of Theorem \ref{th.negative}]
Let us introduce the following three sets:
\begin{equation*}
\begin{split}
&\mathcal{A}=\{C\in(0,\infty): Y(\eta)<0 \ {\rm on} \ l_C \ {\rm for \ any} \ \eta\in\real \ {\rm and} \ \lim\limits_{\eta\to\infty}Y(\eta)=-\infty\},\\
&\mathcal{C}=\{C\in(0,\infty): {\rm there \ exists} \ \eta_0\in\real, Y(\eta_0)>0 \ {\rm on} \ l_C\},\\
&\mathcal{B}=(0,\infty)\setminus(\mathcal{A}\cup\mathcal{C}).
\end{split}
\end{equation*}
We show first that $\mathcal{A}$ is nonempty and open. Let us observe that the limit in the definition of $\mathcal{A}$ implies that there is $\overline{\eta}\in\real$ such that $Y(\eta)<-(\sigma+2)/(p-m)$ for $\eta>\overline{\eta}$, which, together with Lemma \ref{lem.noret}, entail that $l_C$ is an orbit entering the stable node $Q_3$, which also gives the openness of $\mathcal{A}$. Since Lemma \ref{lem.X0} establishes that $l_{\infty}$ satisfies the definition of $\mathcal{A}$, the fact that $Q_3$ is a stable node and the standard continuity with respect to $C$ on the unstable manifold of $P_0$ further imply that there is $C^*>0$ such that $(C^*,\infty)\subseteq\mathcal{A}$. We turn now our attention to the set $\mathcal{C}$. Its openness follows directly by the definition and the continuity with respect to $C$ on the unstable manifold of $P_0$. Moreover, Lemma \ref{lem.Z0} shows that the orbit $l_0$ entirely lies in the half-space $\{Y>0\}$ and enters $P_3$. Once more, the continuity with respect to $C$ gives that there is $C_*>0$ such that $(0,C_*)\subseteq\mathcal{C}$.

Thus, we infer from standard topological arguments that $\mathcal{B}\neq\emptyset$. Pick then $C\in\mathcal{B}$ and let us analyze the orbit $l_C$. The fact that $C\not\in\mathcal{C}$ implies that $Y(\eta)\leq0$ for $\eta\in\real$ on the orbit $l_C$. The fact that $C\not\in\mathcal{A}$ implies that either there exists $\eta_0\in\real$ such that $Y(\eta_0)=0$, or $Y(\eta)<0$ for any $\eta\in\real$ but the limit of $Y(\eta)$ is not equal to $-\infty$, that is, $l_C$ does not end up by entering the node $Q_3$. We rule out the first of these two alternatives.

Assume for contradiction that there exists $\eta_0\in\real$ such that $Y(\eta_0)=0$. Since the orbit $l_C$ is not fully contained in the plane $\{Y=0\}$, one readily infers that $\eta_0$ is a strict local maximum for the function $Y(\eta)$, that is, $Y(\eta_0)=\dot{Y}(\eta_0)=0$ and $Y''(\eta_0)\leq0$. The first two equalities and the second equation of the system \eqref{PSsyst} give that $X(\eta_0)=Z(\eta_0)$. We next compute
\begin{equation*}
\begin{split}
Y''(\eta_0)&=\dot{X}(\eta_0)-(N-2)\dot{Y}(\eta_0)-\dot{Z}(\eta_0)-2mY(\eta_0)\dot{Y}(\eta_0)\\
&+\frac{p-m}{\sigma+2}(X(\eta_0)\dot{Y}(\eta_0)+\dot{X}(\eta_0)Y(\eta_0))\\
&=X(\eta_0)(2-(m-1)Y(\eta_0))-Z(\eta_0)(\sigma+2+(p-m)Y(\eta_0))\\
&=2(X(\eta_0)-Z(\eta_0))-\sigma Z(\eta_0)=-\sigma Z(\eta_0)>0,
\end{split}
\end{equation*}
and we reach a contradiction with the fact that $Y''(\eta_0)\leq0$. Hence, it is not possible that $Y(\eta_0)=0$ for some $\eta_0\in\real$ on the orbit $l_C$ with $C\in\mathcal{B}$. This implies that $Y(\eta)<0$ for any $\eta\in\real$, but $l_C$ does not enter the node $Q_3$, which means that $Y(\eta)$ is bounded from below. Indeed, an application of Lemma \ref{lem.noret} implies that
$$
-\frac{\sigma+2}{p-m}\leq Y(\eta)<0, \qquad \eta\in\real.
$$
We observe once more that $X(\eta)$ is increasing on $l_C$ and thus there is
$$
X_{\infty}=\lim\limits_{\eta\to\infty}X(\eta)>0.
$$
If $X_{\infty}<\infty$, then the $\omega$-limit set of the orbit $l_C$ would be either a (finite) critical point with $X=X_{\infty}$, $Y<0$ (and such a point does not exist), or a critical point with $Z(\eta)\to\infty$, $Y(\eta)$ bounded, $X(\eta)$ bounded (which means it would be the critical point $Q_4$, in contradiction with Lemma \ref{lem.Q4}), or a closed orbit contained in the plane $\{X=X_{\infty}\}$. The latter is also not possible, as such an orbit would be itself a solution to the system \eqref{PSsyst} according to \cite[Theorem 2, Section 3.2]{Pe}, but the condition $X=X_{\infty}\in(0,\infty)$ and the first equation of \eqref{PSsyst} lead to $Y=2/(m-1)>0$, which is a contradiction.

We infer thus that $X_{\infty}=+\infty$. The boundedness of $Y(\eta)$ then gives that $Y(\eta)/X(\eta)\to0$ as $\eta\to\infty$. Since $X$ is a monotone function of $\eta$, we are thus allowed to define a function $Z(X)$ along the orbit $l_C$, which solves the differential equation
$$
\frac{dZ}{dX}=\frac{Z(\sigma+2+(p-m)Y)}{X(2-(m-1)Y)}<\frac{(\sigma+2)Z}{2X},
$$
since $Y<0$ on the orbit $l_C$. By integration and comparison, we infer that there is $C>0$ given by the initial condition such that
$$
Z(\eta)<CX(\eta)^{(\sigma+2)/2}, \qquad \eta\in\real,
$$
hence
$$
\lim\limits_{\eta\to\infty}\frac{Z(\eta)}{X(\eta)}<\lim\limits_{\eta\to\infty}CX(\eta)^{\sigma/2}=0,
$$
as we already have shown that $X(\eta)\to\infty$, while $\sigma<0$. We have thus deduced that the orbit $l_C$ enters the critical point at infinity $Q_1$ (characterized precisely by the conditions $X\to\infty$, $Z/X\to0$ and $Y/X\to0$), as claimed.
\end{proof}

\section{$\sigma\geq0$ small. Existence and multiplicity}\label{sec.mult}

This section is devoted to the proof of Theorem \ref{th.mult}. Thus, throughout this section, we assume that $N\geq2$ and $\sigma\geq0$. We begin by recalling the following result in the limit case $p=m$, that serves as a starting point for the proof. Let us recall that the orbits on the unstable manifold of $P_0$ form a one-parameter family $(l_C)_{C\geq0}$ defined in \eqref{manifP0}.
\begin{lemma}\label{lem.pequalm}
Let $p=m$ and $N\geq2$. Then there exists $\sigma_1>0$ such that, for any $\sigma\in[0,\sigma_1)$, there is $C(\sigma)>0$ such that the orbits $l_C$ with $C\in(0,C(\sigma))$, together with the unique orbit contained in the unstable manifold of the critical point $P_3$, present an infinite number of oscillations around the plane $\{Y=0\}$, that is, changes of sign of $Y(\eta)$ for $\eta\in\real$.
\end{lemma}
\begin{proof}
This has been proved recently in \cite[Theorem 1.4, Theorem 1.5, Section 2.3]{IS22} and their proofs (contained in \cite[Section 6 and Section 7]{IS22}), noticing that the critical points $P_2$, respectively $Q_1$, in the notation of \cite{IS22}, correspond to the points $P_3$, respectively $P_0$, in our current notation. For the proof to be complete, since in \cite{IS22} we have used different phase space variables, we notice that, with respect to the variables $(x,y,z)$ employed in \cite[Lemma 3.1]{IS22}, which are
$$
x(\xi)=\frac{1}{\sqrt{m(m-1)}}\xi f^{(1-m)/2}(\xi), \qquad y(\xi)=\frac{\xi f'(\xi)}{f(\xi)}, \qquad z(\xi)=\sqrt{\frac{m-1}{m}}\xi^{\sigma+1}f(\xi)^{(m-1)/2},
$$
our variables $(X,Y,Z)$ of the system \eqref{PSsyst} relate in the following way:
\begin{equation}\label{changepm}
X(\xi)=x^2(\xi), \qquad Y(\xi)=y(\xi), \qquad Z(\xi)=x(\xi)z(\xi).
\end{equation}
Thus, the critical point denoted by $Q_1$ in \cite[Lemma 3.1]{IS22} matches with our critical point $P_0$, while the critical point denoted by $P_2$ in \cite{IS22} matches with our critical point $P_3$, as one can readily check from \eqref{changepm}. Moreover, since the variable $Y$ is the same in both systems, any oscillation with respect to the plane $\{y=0\}$ in the notation of \cite{IS22} corresponds to an oscillation with respect to the plane $\{Y=0\}$ in our current setting.
\end{proof}

\noindent \textbf{Remark.} As an outcome of the results in \cite{IS22, IL22}, we have an upper estimate for $\sigma_1$, namely
$$
\sigma_1<\left\{\begin{array}{ll}\frac{N(m-1)}{m+1}, & N\in\{2,3\}, \\ \frac{2(N-1)(m-1)}{3m+1}, & N\geq4.\end{array}\right.
$$

With these preparations, we are in a position to prove what we consider to be the central result of this work, namely Theorem \ref{th.mult}. Fix $\sigma\geq0$.
\begin{proof}[Proof of Theorem \ref{th.mult}]
The proof will be divided into several steps, for the reader's convenience. For the whole proof, fix $p\in(m,p_c(\sigma))$, where $p_c(\sigma)$ has been introduced in \eqref{crit.p}.

\medskip

\noindent \textbf{Step 1. A surface built over the trajectory entering $Q_{\gamma_0}$.} We begin with the following geometric construction in the phase space associated to the system \eqref{PSsyst}. Let $r_0$ be the unique trajectory entering the critical point $Q_{\gamma_0}$ according to Lemma \ref{lem.Qg}. Since, as shown in the first step of the proof of Lemma \ref{lem.noret}, the plane $\{Y=-(\sigma+2)/(p-m)\}$ has a negative direction of the flow at every point, it follows that $Y(\eta)>-(\sigma+2)/(p-m)$ for any $\eta\in\real$ on the trajectory $r_0$. We have thus the following alternative:

$\bullet$ either on the trajectory $r_0$ we have
\begin{equation}\label{alter1}
-\frac{\sigma+2}{p-m}<Y(\eta)<\frac{2}{m-1}, \qquad {\rm for \ any} \ \eta\in\real, \qquad \limsup\limits_{\eta\to-\infty}Y(\eta)<\frac{2}{m-1}.
\end{equation}
In this case, $X(\eta)$ and $Z(\eta)$ are both increasing functions of $\eta$, as it follows from the first and third equation of the system \eqref{PSsyst}. This monotonicity, together with arguments as in the proof of \cite[Proposition 4.10]{ILS24}, show that $r_0$ stems from one of the critical points $P_0$ or $P_1$ (since the last condition in \eqref{alter1} avoids $P_3$), that is, $X(\eta)\to0$ as $\eta\to-\infty$. Moreover, on the one hand, on the projection of $r_0$ on the plane $\{Z=0\}$, the function $X(\eta)$ is invertible and we can see this curve as the graph of a function $Y=\psi(X)$. On the other hand, above every point $(X,\psi(X))$, there is a unique value of $Z=Z(X)$ such that $(X,\psi(X),Z(X))\in r_0$, due to the monotonicity of $Z$ along the trajectory $r_0$. We thus build the surface $\mathcal{S}$ to be the cylinder $(X,\psi(X),Z)$ raised vertically over the curve $Y=\psi(X)$ which is the projection of $r_0$ onto the plane $\{Z=0\}$. Obviously, $r_0$ is contained in $\mathcal{S}$.
By the inverse function theorem, we readily find that
$$
\psi'(X)=\frac{X-(N-2)\psi(X)-Z(X)-m\psi(X)^2+(p-m)/(\sigma+2)X\psi(X)}{X(2-(m-1)\psi(X))},
$$
hence, the flow of the system \eqref{PSsyst} across this surface $\mathcal{S}$, with normal vector
$$
\overline{n}(X,Y,Z)=(-\psi'(X),1,0),
$$
is given by the sign of the expression
\begin{equation}\label{flow.alter1}
\begin{split}
F(X,Z)&=-X(2-(m-1)\psi(X))\psi'(X)+X-(N-2)\psi(X)-Z\\&-m\psi(X)^2+\frac{p-m}{\sigma+2}X\psi(X)=Z(X)-Z.
\end{split}
\end{equation}
This proves that the direction of the flow across $\mathcal{S}$ is separated by the trajectory $r_0$: above it (that is, with $Z>Z(X)$) orbits cross the surface pointing into the negative direction and below it (that is, with $Z<Z(X)$) into the positive direction of the $Y$ variable. Consequently, there cannot be trajectories of the system tangent to $\mathcal{S}$: it is easy to show that a tangency point must lie at $Z=Z(X)$, that is, on $r_0$, and this contradicts the local uniqueness of solutions, since $r_0$ is itself a solution to \eqref{PSsyst}. Notice that, for the limiting case $\sigma=0$, the trajectory $r_0$ is exactly the line
$$
X(\eta)=Z(\eta), \qquad Y(\eta)=0, \qquad \eta\in\real,
$$
and the surface $\mathcal{S}$ coincides with the plane $\{Y=0\}$.

\medskip

$\bullet$ or there is $\eta_0\in\real\cup\{-\infty\}$ such that on the trajectory $r_0$ we have
$$
Y(\eta_0)=\frac{2}{m-1}, \qquad -\frac{\sigma+2}{p-m}<Y(\eta)<\frac{2}{m-1}, \qquad {\rm for \ any} \ \eta\in(\eta_0,\infty).
$$
In this case, we modify the definition of the surface $\mathcal{S}$ to be the cylinder $(X,\psi(X),Z)$ constructed exactly as above over the part $\{(X,Y,Z)(\eta):\eta\in(\eta_0,\infty)\}$ of the trajectory $r_0$, and extended with the portion of the plane $\{Y=2/(m-1), 0\leq X\leq X(\eta_0)\}$. The flow of the system \eqref{PSsyst} across the plane $\{Y=2/(m-1)\}$, with normal vector $(0,1,0)$, is given by the sign of the expression
\begin{equation}\label{flow.alter2}
\begin{split}
G(X,Z)&=X\left(1+\frac{2(p-m)}{(m-1)(\sigma+2)}\right)-\frac{2(N-2)}{m-1}-\frac{4m}{(m-1)^2}-Z\\&=\frac{L}{(\sigma+2)(m-1)}(X-X(P_3))-Z,
\end{split}
\end{equation}
where $X(P_3)$ is defined in \eqref{zp2}. Thus, this flow has a unique direction (negative one with respect to the coordinate $Y$) provided $X\leq X(P_3)$.

Since for $\sigma=0$, the trajectory $r_0$ is contained in the plane $\{Y=0\}$, connecting to $P_0$ (which is a saddle point), it follows from the monotonicity of $\eta\mapsto X(\eta)$, the fact that $P_3\neq P_0$ for $\sigma=0$ and standard continuity with respect to $\sigma$ that there is $\sigma_2>0$ such that, for any $\sigma\in(0,\sigma_2)$, the trajectory $r_0$ fulfills either the first alternative above, or the second alternative with $X(\eta_0)\leq X(P_3)$. Thus, for any $\sigma\in(0,\sigma_2)$, the directions of the flow given by \eqref{flow.alter1} and \eqref{flow.alter2} are in force, and thus no trajectory of the system \eqref{PSsyst} can be tangent to the surface $\mathcal{S}$, in any of the two alternatives. Fix from now on
$$
\sigma_0:=\min\{\sigma_1,\sigma_2\}>0,
$$
with $\sigma_1$ introduced in Lemma \ref{lem.pequalm}. The most important geometric property of the surface $\mathcal{S}$ constructed above (in both alternatives) is the fact that it splits the space into two disjoint regions. Thus, for any $\sigma\in(0,\sigma_0)$, let
\begin{equation*}
\begin{split}
&\mathcal{S}_{-}=\left\{(X,Y,Z)\in\real^3: X\geq0, Z\geq0, Y<\min\left\{\psi(X),\frac{2}{m-1}\right\}\right\}, \\
&\mathcal{S}_+=\left\{(X,Y,Z)\in\real^3: X\geq0, Z\geq0, Y>\min\left\{\psi(X),\frac{2}{m-1}\right\}\right\},
\end{split}
\end{equation*}
the regions of the space limited by $\mathcal{S}$. 

We say that a trajectory $(X,Y,Z)(\eta)$ of the system \eqref{PSsyst} \textbf{has an oscillation} with respect to the surface $\mathcal{S}$ if there are $\eta_*<\eta^*\in\real$ such that
$$
(X,Y,Z)(\eta_*)\in\mathcal{S}_{-}, \qquad (X,Y,Z)(\eta^*)\in\mathcal{S}_+,
$$
and there is a unique $\eta\in(\eta_*,\eta^*)$ such that $(X,Y,Z)(\eta)\in\mathcal{S}$.

\medskip

\noindent \textbf{Step 2. No trajectory $l_C$ with $C>0$ stays forever in $\mathcal{S}_+$}. Assume for contradiction that there is $C\in(0,\infty)$ such that $(X,Y,Z)(\eta)\in\mathcal{S}_+$ for any $\eta\in\real$ on the trajectory $l_C$. The third equation of \eqref{PSsyst} gives then that $\eta\mapsto Z(\eta)$ is an increasing function and thus there is $Z_{\infty}=\lim\limits_{\eta\to\infty}Z(\eta)$. Since, by the study in Sections \ref{sec.local} and \ref{sec.infty}, all the critical points (namely, $P_3$, $Q_5$, $Q_2$, $Q_4$ and $Q_{\gamma}$) lying in $\mathcal{S}_+\cup\mathcal{S}$ cannot be the limit as $\eta\to\infty$ of $l_C$ for $C>0$, it then follows that

$\bullet$ either $Z_{\infty}\in(0,\infty)$ and then $l_C$ has an $\omega$-limit contained in the plane $\{Z=Z_{\infty}\}$. In this case, since the $\omega$-limit is itself a trajectory of the system according to \cite[Section 3.2]{Pe}, it follows from the third equation in \eqref{PSsyst} that $Y=-(\sigma+2)/(p-m)$ along it, whence it lies in the region $\mathcal{S}_{-}$ and we reach a contradiction.

$\bullet$ or $Z_{\infty}=\infty$. Since $Y(\eta)$ is bounded from below, we readily infer that $X(\eta)\to\infty$ as $\eta\to\infty$ along the trajectory $l_C$; indeed, if this is not the case, then we easily derive from the first and second equation in \eqref{PSsyst} that both $X(\eta)$ and $Y(\eta)$ would be dominated by $Z(\eta)$ as $\eta\to\infty$ and the limit would be the critical point $Q_4$, contradicting Lemma \ref{lem.Q4}. We thus conclude that the $\omega$-limit of $l_C$ is contained in the invariant plane $\{x=0\}$ of the system \eqref{PSinf1}. It follows then by Lemma \ref{lem.w0} that such limit, if not a single point, would be contained in the critical line $\{x=y=0\}$. We are thus exactly in the same situation as in the final part of Lemma \ref{lem.ret} and we can continue as there to prove that the trajectory $l_C$ under consideration connects to one of the critical points $Q_{\gamma_0}$ or $Q_1$. We conclude this step by recalling that, in the former of these two cases, the trajectory would be fully contained (at least for $\eta\in(\eta_0,\infty)$ for some $\eta_0\in\real$) in the surface $\mathcal{S}$ by its construction, reaching thus a contradiction, while in the latter, for the point $Q_1$ we have $Y=y/x\to-(\sigma+2)/(p-m)$ by Lemma \ref{lem.Q1Q5}, thus the orbit would enter $\mathcal{S}_{-}$ at some point.

\medskip

\noindent \textbf{Step 3. Existence of one profile with the desired behavior}. Since for $\sigma=0$ the existence of a first self-similar profile as in Theorem \ref{th.mult} for any $p\in(m,p_s(0))$ is a well-known result (see for example, \cite[Theorem 4, Section 1.4, Chapter IV]{S4}), for this step we specialize to $\sigma>0$. Fix thus $\sigma\in(0,\sigma_0)$. Taking $p$ (or equivalently $\mu=p-m$) as a parameter in the system \eqref{PSsyst}, we observe that the parameter $\mu$ in a neighborhood of zero does not affect the equilibrium state of $P_0$ and $P_3$, as for the former, it only appears as a parameter in the higher order terms, while for the latter, the eigenvalues of $M(P_3)$ change in a continuous way with respect to $p$ (or $\mu=p-m$) but without changing their character in a neighborhood of $\mu=0$. We infer from \cite[Theorem 5.14]{Shilnikov} that locally, in a neighborhood of the respective critical points, their unstable manifolds vary in a continuous way with respect to $\mu$, which is the same as with respect to $p$ in a neighborhood of $p=m$. Thus, by joining this local continuity of the unstable manifolds in a neighborhood of the critical points, with the standard continuity argument on finite intervals of the independent variable $\eta\in[\eta_{-},\eta_{+}]$ for any $-\infty<\eta_{-}<\eta_{+}<+\infty$, and with the infinite number of oscillations of orbits in the limit case $p=m$ given in Lemma \ref{lem.pequalm}, we deduce that there exists $p_1(\sigma)$ with $m<p_1(\sigma)\leq p_c(\sigma)$ such that for any $p\in(m,p_1(\sigma))$, the unique orbit stemming from $P_3$ on its unstable manifold and orbits $l_C$ for $C\in(0,D_1(p,\sigma))$ for some $D_1(p,\sigma)>0$ (depending on $p$ and $\sigma$) on the unstable manifold of $P_0$ present at least one oscillation (in the sense of the definition at the end of Step 1) with respect to the surface $\mathcal{S}$ constructed in Step 1.

We infer from Step 2 that there is $\eta_{0,C}\in[-\infty,\infty)$ (where the case $\eta_{0,C}=-\infty$ applies if the orbit $l_C$ directly goes out into $\mathcal{S}_{-}$) such that, if $\eta_{0,C}>-\infty$, we have
\begin{equation*}
\begin{split}
&(X,Y,Z)(\eta)\in\mathcal{S}_+, \qquad {\rm for} \ \eta\in(-\infty,\eta_{0,C}), \qquad (X,Y,Z)(\eta_{0,C})\in\mathcal{S},\\
&(X,Y,Z)(\eta)\in\mathcal{S}_-, \qquad {\rm for} \ \eta\in(\eta_{0,C},\eta_{0,C}+\delta), \qquad \delta>0 \ {\rm small}.
\end{split}
\end{equation*}
Fix thus $p\in(m,p_1(\sigma))$ and introduce the following three sets with respect to the orbits $(l_C)_{C\geq0}$ defined in \eqref{manifP0}:
\begin{equation}\label{sets1}
\begin{split}
&\mathcal{A}_1=\left\{C\in(0,\infty): (X,Y,Z)(\eta)\in\mathcal{S}_-, \ {\rm if} \ \eta\in(\eta_{0,C},\infty), \ {\rm and \ there \ is} \right.\\&\left.\eta_{1,C}\in\real, \ Y(\eta_{1,C})<-\frac{\sigma+2}{p-m}\right\},\\
&\mathcal{C}_1=\{C\in(0,\infty): l_C \ {\rm has \ at \ least \ one \ oscillation \ with \ respect \ to} \ \mathcal{S} \ {\rm on} \ (\eta_{0,C},\infty)\},\\
&\mathcal{B}_1=(0,\infty)\setminus(\mathcal{A}_1\cup\mathcal{C}_1).
\end{split}
\end{equation}
According to Lemma \ref{lem.noret}, if $C\in\mathcal{A}_1$, then the orbit $l_C$ enters the stable node $Q_3$. The definition of $\mathcal{A}_1$ entails the openness of the set $\mathcal{A}_1$, while the fact that the orbit $l_{\infty}$ enters $Q_3$ with $Y(\eta)<0$ for any $\eta\in\real$ (as shown in Lemma \ref{lem.X0}) and the continuity with respect to $C$ on the unstable manifold of $P_0$ give the non-emptiness of $\mathcal{A}_1$, which contains an interval of the form $(C^*,\infty)$. With respect to $\mathcal{C}_1$, this set is open by the definition of an oscillation, while its non-emptiness is granted by the definition of $p_1(\sigma)$ (and it contains an interval of the form $(0,C_*)$). A standard argument of topology gives then that $\mathcal{B}_1\neq\emptyset$. Take then $C_1:=\sup\mathcal{C}_1<\infty$, and notice that $C_1\in\mathcal{B}_1$. Then, on the orbit $l_{C_1}$, we have on the one hand that $-(\sigma+2)/(p-m)<Y(\eta)$ for any $\eta\in(\eta_{0,C_1},\infty)$ (since $l_{C_1}$ cannot be tangent to the plane $\{Y=-(\sigma+2)/(p-m)\}$), and on the other hand,
$$
(X,Y,Z)(\eta)\in\mathcal{S}_{-}, \qquad {\rm for \ any} \ \eta\in(\eta_{0,C_1},\infty),
$$
since by Step 1, this limiting orbit $l_{C_1}$ cannot be tangent to $\mathcal{S}$. In particular, by the construction of $\mathcal{S}$, the condition \eqref{interm22} is fulfilled for any $\eta\in(\eta_{0,C_1},\infty)$. Lemma \ref{lem.ret} then entails that $l_{C_1}$ either enters $Q_1$ or $Q_{\gamma_0}$. But the latter is discarded, since there is a unique orbit entering $Q_{\gamma_0}$, which is fully contained in $\mathcal{S}$, so that $l_{C_1}$ connects to $Q_1$, as desired.

We plot in Figure \ref{fig1} several orbits presenting either none or one oscillation, together with the profiles $f(\xi)$ corresponding to them, as an outcome of numerical experiments that allow us to visually explain the proof of the existence step.

\begin{figure}[ht!]
  \begin{center}
  \subfigure[Various trajectories in the sets $\mathcal{A}_1$ and $\mathcal{C}_1$]{\includegraphics[width=7.5cm,height=6cm]{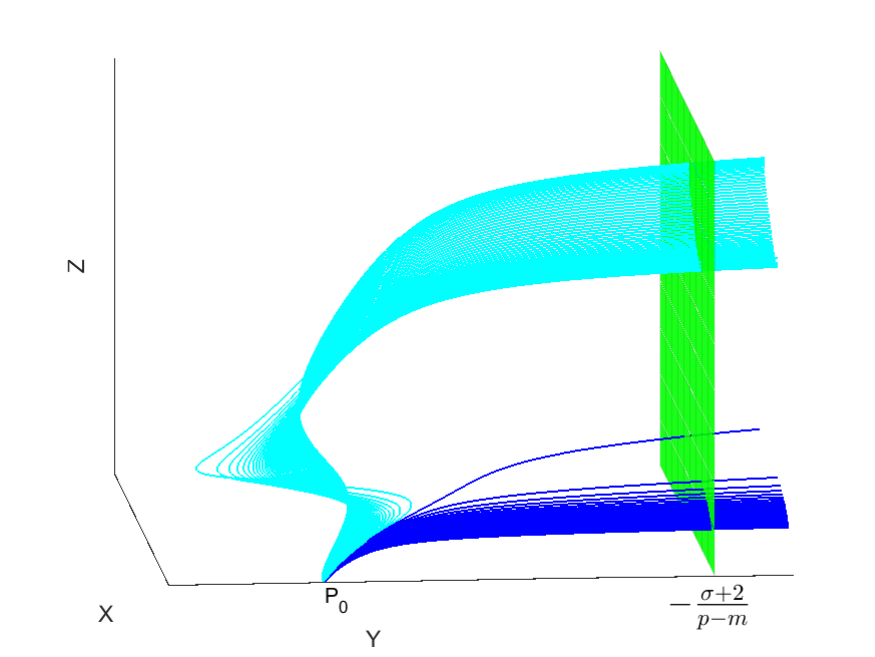}}
  \subfigure[Profiles corresponding to the previous trajectories]{\includegraphics[width=7.5cm,height=6cm]{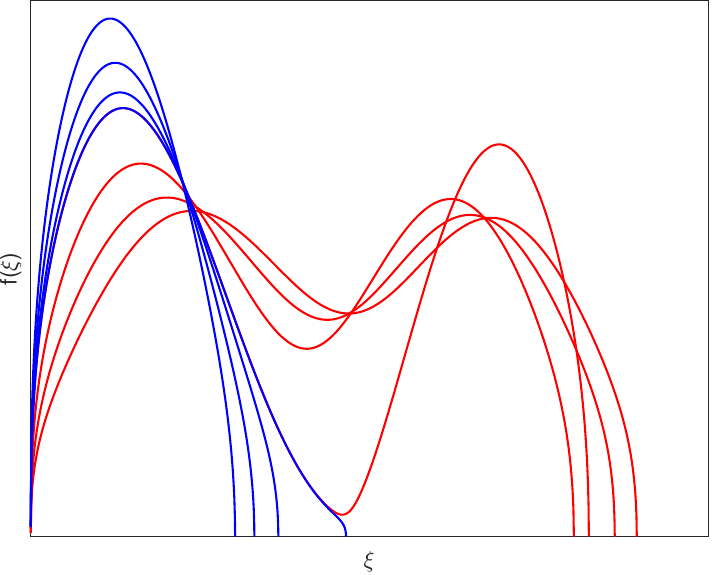}}
  \end{center}
  \caption{Various trajectories presenting none or one oscillation, seen in the phase space and in profiles. Experiments for $m=2$, $N=5$, $p=2.1$ and $\sigma=0.5$.}\label{fig1}
\end{figure}

\medskip

\noindent \textbf{Step 4. Profiles corresponding to orbits with one oscillation}. Fix now $\sigma\in[0,\sigma_0)$. By applying the same continuity argument as in Step 3 but eventually increasing the finite endpoint $\eta_+<\infty$ up to which the argument of continuity is considered, we obtain that there is $p_2(\sigma)\in(m,p_1(\sigma))$ such that, for any $p\in(m,p_2(\sigma))$, there exists $D_2(p,\sigma)>0$ such that all the orbits $l_C$ as in \eqref{manifP0} with $0<C<D_2(p,\sigma)$ and the unique orbit contained in the unstable manifold of $P_3$ present at least \textbf{two oscillations} with respect to the surface $\mathcal{S}$: that is, there are $-\infty<\eta_{*,1}<\eta^{*}_1<\eta_{*,2}<\eta^{*}_2<\infty$ such that
\begin{equation*}
(X,Y,Z)(\eta_{*,i})\in\mathcal{S}_+, \qquad (X,Y,Z)(\eta^{*}_i)\in\mathcal{S}_{-}, \qquad i=1,2,
\end{equation*}
and there is a unique intersection between the trajectory and $\mathcal{S}$ for $\eta\in(\eta_{*,1},\eta^{*}_1)$, for $\eta\in(\eta^{*}_1,\eta_{*,2})$ and also for $\eta\in(\eta_{*,2},\eta^{*}_2)$. We next split the open set $\mathcal{C}_1$ introduced in \eqref{sets1} into the following three subsets:
\begin{equation}\label{sets2}
\begin{split}
&\mathcal{A}_2=\left\{C\in\mathcal{C}_1: l_C \ {\rm has \ exactly \ one \ oscillation \ and \ there \ is} \right.\\&\left. \eta_{2,C}\in\real, \ Y(\eta_{2,C})<-\frac{\sigma+2}{p-m}\right\},\\
&\mathcal{C}_2=\{C\in\mathcal{C}_1: l_C \ {\rm has \ at \ least \ two \ oscillations \ with \ respect \ to} \ \mathcal{S}\},\\
&\mathcal{B}_2=\mathcal{C}_1\setminus(\mathcal{A}_2\cup\mathcal{C}_2).
\end{split}
\end{equation}
Once more, we infer from Lemma \ref{lem.noret} and the stability of the node $Q_3$ that $\mathcal{A}_2$ is an open set, while $\mathcal{C}_2$ is an open set by definition, which is also non-empty by the choice of $p\in(m,p_2(\sigma))$, in particular $\mathcal{C}_2$ contains an interval of the form $(0,D_2(p,\sigma))$ as explained above. The only new difficulty comes from the non-emptiness of the set $\mathcal{A}_2$. According to Step 3, letting $C_1=\sup\mathcal{C}_1$, then the trajectory $l_{C_1}$ connects $P_0$ to $Q_1$ without any oscillation with respect to the surface $\mathcal{S}$. Applying the generalized boundary value theorem \cite[Theorem 5.11]{Shilnikov} in a sufficiently small neighborhood $\mathcal{U}$ of the saddle point $Q_1$, we infer that there exists $\overline{C}\in(0,C_1)$ such that, for any $C\in(\overline{C},C_1)$ the orbit $l_{C}$ enters the neighborhood $\mathcal{U}$ of $Q_1$ and then follows the trajectory of orbits stemming from $Q_1$ on its unstable manifold. Since $C_1=\sup\mathcal{C}_1$, we can pick $C'\in(\overline{C},C_1)\cap\mathcal{C}_1$. The orbit $l_{C'}$ then will follow the trajectories given by Lemma \ref{lem.w0} presenting a single oscillation with respect to the surface $\mathcal{S}$ (by definition of $\mathcal{C}_1$) in a very small strip very close to the plane $\{w=0\}$. Hence $C'\in\mathcal{A}_2$ and thus $\mathcal{A}_2$ is non-empty and open. Since $\mathcal{C}_1$ is open, it then follows that $\mathcal{B}_2$ is non-empty. Completely similar considerations as in the end of Step 3 and based on Lemma \ref{lem.ret} prove that $l_C$ connects to the critical point $Q_1$ after a unique oscillation with respect to the surface $\mathcal{S}$, for $C\in\mathcal{B}_2$, as claimed (and in particular for $C_2=\sup\mathcal{C}_2$). This step also applies for $\sigma=0$.

\medskip

\noindent \textbf{Step 5. Profiles corresponding to orbits with $n$ oscillations, $n\geq2$}. This step extends, by complete induction, the analysis performed in Step 4, so that we will be rather brief. Assume that there exists $p_n(\sigma)$ such that, for any $p\in(m,p_n(\sigma))$, there exist at least an orbit connecting $P_0$ to $Q_1$ and having exactly $n-1$ oscillations with respect to the surface $\mathcal{S}$, and we have defined sets $\mathcal{A}_n$, $\mathcal{C}_n$ and $\mathcal{B}_n$ similarly as in \eqref{sets2}, by assuming that elements in $\mathcal{A}_n$ have exactly $n-1$ oscillations and elements in $\mathcal{C}_n$ have at least $n$ oscillations. By extending eventually the finite endpoint $\eta_+<\infty$ of the interval on which the continuity argument with respect to the limiting case $p=m$ (according to Lemma \ref{lem.pequalm} and \cite[Theorem 5.14]{Shilnikov}) is applied, we deduce the existence of $p_{n+1}(\sigma)\in(m,p_n(\sigma))$ and of a constant $D_{n+1}(p,\sigma)>0$ such that, for any $p\in(m,p_{n+1}(\sigma))$, all the orbits $l_C$ defined in \eqref{manifP0} with $C\in(0,D_{n+1}(p,\sigma))$ and the unique orbit on the unstable manifold of $P_3$ present at least $n+1$ oscillations with respect to $\mathcal{S}$. We then split the open set $\mathcal{C}_n$ (defined in the inductive step, non-empty and open) into the following three sets:
\begin{equation}\label{setsn}
\begin{split}
&\mathcal{A}_{n+1}=\left\{C\in\mathcal{C}_n: l_C \ {\rm has \ exactly} \ n \ {\rm oscillations \ and \ there \ is} \right.\\&\left. \eta_{n+1,C}\in\real, \ Y(\eta_{n+1,C})<-\frac{\sigma+2}{p-m}\right\},\\
&\mathcal{C}_{n+1}=\{C\in\mathcal{C}_n: l_C \ {\rm has \ at \ least} \ n+1 \ {\rm oscillations \ with \ respect \ to} \ \mathcal{S}\},\\
&\mathcal{B}_{n+1}=\mathcal{C}_n\setminus(\mathcal{A}_{n+1}\cup\mathcal{C}_{n+1}).
\end{split}
\end{equation}
Arguing as in Step 3, we infer that $\mathcal{C}_{n+1}$ is open and non-empty (contains at least the interval $(0,D_{n+1}(p,\sigma))$, by construction), while Lemmas \ref{lem.w0} and \ref{lem.noret} and the boundary value problem \cite[Theorem 5.11]{Shilnikov} in a neighborhood of $Q_1$ ensure that $\mathcal{A}_{n+1}$ is also non-empty and open. Since $\mathcal{C}_n$ was open, we conclude by similar arguments as in the end of Step 3 that there are elements $C\in\mathcal{B}_{n+1}$ and that the corresponding trajectories $l_C$ connect $P_0$ to $Q_1$ after exactly $n$ oscillations with respect to the surface $\mathcal{S}$. This proves the first part of the statement of Theorem \ref{th.mult}.

We plot in Figure \ref{fig2} a number of trajectories presenting several different numbers of oscillations obtained through a numerical experiment based on a shooting of the orbits $l_C$ for a big amount of values of the parameter $C>0$. In the second figure, we also plot several profiles corresponding to (a part of) these trajectories.

\begin{figure}[ht!]
  \begin{center}
  \subfigure[Trajectories $l_C$ with different numbers of oscillations]{\includegraphics[width=7.5cm,height=6cm]{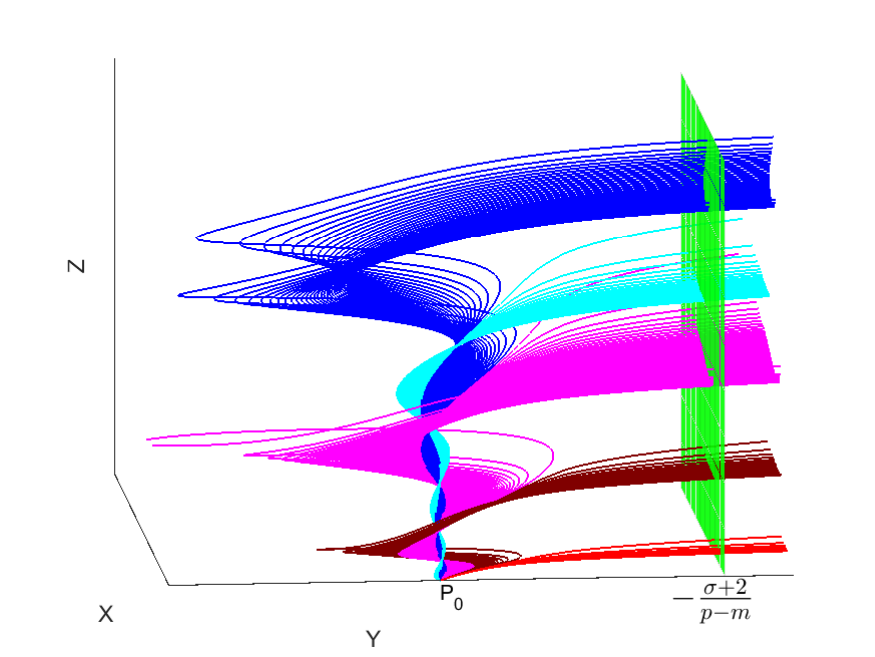}}
  \subfigure[Profiles corresponding to some of the previous trajectories]{\includegraphics[width=7.5cm,height=6cm]{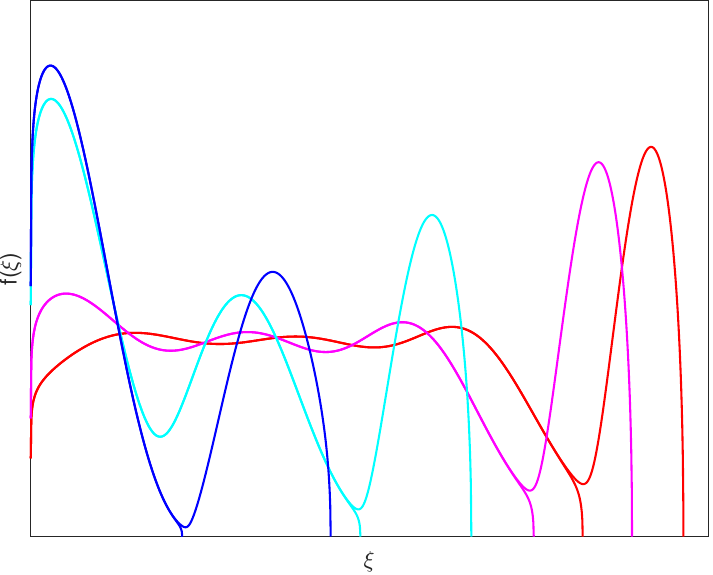}}
  \end{center}
  \caption{A shooting of trajectories $l_C$ presenting different numbers of oscillations. Experiments for $m=2$, $N=5$, $p=2.1$ and $\sigma=0.1$.}\label{fig2}
\end{figure}

\medskip

\noindent \textbf{Step 6. Profiles with dead-core and $n$ oscillations}. We show very briefly in this step that a similar construction as in Steps 3-5 can be performed starting from $Q_5$ instead of $P_0$. We then work in the variables $(x,y,z)$ of the system \eqref{PSinf1} and setting $w=xz$. On the one hand, Lemma \ref{lem.w0} proves that the orbit starting from $Q_5$ contained in the plane $\{w=0\}$ connects directly, with a decreasing trajectory in the $y$ variable, to the stable node $Q_3$. We further infer by Lemma \ref{lem.Z0} that the orbit starting from $Q_5$ contained in the plane $\{z=0\}$ (which is the same as $\{w=0\}$) connects to the saddle $P_3$ on its two-dimensional stable manifold fully contained in the plane $\{z=0\}$. An application of \cite[Theorem 2.9]{Shilnikov} in a neighborhood of the saddle point $P_3$ shows that the trajectories on the two-dimensional unstable manifold of $Q_5$ and "close" to the trajectory contained in $\{z=0\}$ (in a similar sense as the orbits $l_C$ with $C\in(0,D_{n}(p,\sigma))$ for the unstable manifold of $P_0$) will present as many oscillations with respect to the surface $\mathcal{S}$ as the unique trajectory on the unstable manifold of $P_3$. We are thus for the two-dimensional unstable manifold of $Q_5$ in exactly the same position as we were in Steps 3-5 for the two-dimensional unstable manifold of $P_0$. We can thus perform a completely similar induction argument as in Steps 3-5 above with respect to the number of oscillations with respect to the surface $\mathcal{S}$ and then recall that the local behavior of profiles contained in the orbits stemming from $Q_5$ is a dead-core according to Lemma \ref{lem.Q1Q5}, to complete the proof of the second part of the statement of Theorem \ref{th.mult}. 

\medskip

\noindent \textbf{Step 7. Quantitative improvement for $\sigma=0$}. In the case $\sigma=0$, the surface $\mathcal{S}$ coincides with the plane $\{Y=0\}$ and the orbit $r_0$ entering $Q_{\gamma_0}$ is represented by the line $\{X=Z\}$ inside the plane $\{Y=0\}$ in the system \eqref{PSsyst}, which gives the constant profile
$$
f(\xi)=\left(\frac{1}{p-1}\right)^{1/(p-1)}.
$$
The linearization performed in a neighborhood of this constant solution in \cite[Section 1.4, Chapter IV, pp. 191-193]{S4} in any space dimension $N\geq1$ shows that there are solutions to \eqref{SSODE} corresponding to orbits on the unstable manifold of $P_0$ in our notation, presenting exactly $k$ different maxima and minima, provided $p_{k+1}(0)\leq p<p_k(0)$, where $p_k(0)=(mk-1)/(k-1)$. Since maxima and minima on the profile correspond to oscillations with respect to the plane $\{Y=0\}$ in our setting, taking into account that the orbit from $P_0$ contained in the plane $\{X=0\}$ enters $Q_3$ with a decreasing trajectory for any $p\in(m,p_s(0))$, as established in Lemma \ref{lem.X0}, we can repeat Steps 3-5 above with these precise numbers $p_k(0)$. In particular, let us notice that
$$
p_k(0)<p_s(0) \qquad {\rm iff} \qquad k>K(m,N)=\frac{mN-N+2m+2}{4m},
$$
which shows that for any $p\in(m,p_s(0))$, we have at least $k$ different profiles, where $k$ is the closest integer inferior to the number $K(m,N)$ defined in \eqref{KMN}. A completely similar argument has been employed (for larger values of $p$) in the recent paper \cite{IS24}, where full details are given for $\sigma=0$.

\medskip

\noindent \textbf{Step 8. Profiles with local behavior \eqref{beh.P3} as $\xi\to0$}. According to Lemma \ref{lem.P3}, we are looking in this step for connections from $P_3$ to $Q_1$. Similarly as in the Steps 3-5 above, at every change (as $p$ moves decreasingly towards $m$) of the number of oscillations with respect to the surface $\mathcal{S}$ of the trajectory of the unique orbit contained in the unstable manifold of $P_3$, there is one value of $p$ for which this orbit connects $P_3$ to $Q_1$. Since we do not have an estimate of the number of oscillations of this orbit when $p\sim p_s(\sigma)$, but we know that it tends to infinity as $p\to m$, $p>m$, by the continuity with respect to the limiting case $p=m$ as explained in Step 3 above, there is at least one value of $p$ at which the number of oscillations changes and the argument in Step 3 can be applied to get an orbit connecting $P_3$ to $Q_1$.
\end{proof}

\noindent \textbf{Remark.} The fact that the proof is performed for $m<p\leq p_c(\sigma)$ does not mean necessarily that $p_1(\sigma)<p_c(\sigma)$, as the range of existence of the first solution (as deduced in Step 2 of the proof) might be larger. Actually, except for $\sigma=0$, we have no clue related to a precise estimate of the numbers $p_k(\sigma)$ for a given $\sigma\in(0,\sigma_0)$.

\section{Non-existence. Proof of Theorem \ref{th.nonexist}}\label{sec.nonexist}

This section is devoted to the proof of the non-existence result for $\sigma>0$ sufficiently large as stated in Theorem \ref{th.nonexist}. One of the most interesting aspects of this proof is that it involves a mix of two very different techniques: on the one hand, for $m<p\leq p_F(\sigma)$, with $p_F(\sigma)$ the Fujita-type exponent defined in \eqref{Fujita}, we construct a Pohozaev functional and employ it to prove non-existence. On the other hand, for $p_F(\sigma)<p<p_s(\sigma)$, we employ an argument of geometric barriers in the phase space associated to the system \eqref{PSsyst} in form of a combination of planes and surfaces blocking the access of the orbits to the critical point $Q_1$, provided $\sigma$ is sufficiently large. We divide this section into three subsections according to this scheme.

\subsection{A Pohozaev identity}\label{subsec.Poh}

Despite the fact that we are working with radially symmetric solutions in the rest of this work, we will work throughout this section with general solutions, the outcome being a more general non-existence result (including non-radial self-similar solutions) in the range where the Pohozaev identity deduced below applies. We proceed as in \cite{FT00, IL22} and in a first step, pass to self-similar variables by setting
\begin{equation}\label{selfsim.var}
u(x,t)=(T-t)^{-\alpha}v(x(T-t)^{-\beta}), \qquad y=x(T-t)^{-\beta},
\end{equation}
and plug this ansatz into Eq. \eqref{eq1} to deduce after straightforward calculations that $v(y)$ solves the equation
\begin{equation}\label{eq.sim.var}
-\Delta V+\alpha V^{1/m}+\beta y\cdot\nabla V^{1/m}-|y|^{\sigma}V^{p/m}=0, \qquad V(y)=v^m(y).
\end{equation}
Assume now that we are dealing with solutions $v(y)$ satisfying the following decay estimates as $|y|\to\infty$:
\begin{equation}\label{decays1}
v(y)\leq C_1|y|^{-(\sigma+2)/(p-m)}, \qquad |\nabla v(y)|\leq C_2|y|^{-(\sigma+2)/(p-m)-1},
\end{equation}
for some $C_1$, $C_2>0$. We then infer from \eqref{decays1} and \eqref{eq.sim.var} that
\begin{equation}\label{decays2}
\begin{split}
V(y)\leq C_1^m|y|^{-m(\sigma+2)/(p-m)}, \qquad &|\nabla V(y)|\leq C_3|y|^{-m(\sigma+2)/(p-m)-1},\\ &\Delta V(y)\leq C_0|y|^{-(\sigma+2)/(p-m)},
\end{split}
\end{equation}
as $|y|\to\infty$, for some positive constants $C_0$ and $C_3$. The rest of the construction follows several steps, as indicated below.

\medskip

\noindent \textbf{Step 1. Multiply \eqref{eq.sim.var} by $y\cdot\nabla V$}. We compute term by term the outcome, integrating by parts when needed (and we justify at the end of the step that the integrals obtained are convergent). Multiplying the first term in \eqref{eq.sim.var} by $y\cdot\nabla V$ and integrating, we find
\begin{equation}\label{Poh1}
\begin{split}
-\int_{\real^N}\left(y\cdot\nabla V \right)\Delta V\,dy&=\int_{\real^N}\nabla V\cdot\nabla\left(y\cdot\nabla V\right)\,dy=\int_{\real^N}\sum\limits_{i,j=1}^{N}\partial_iV\partial_i\left(x_j\partial_jV \right)\,dy\\
&=\int_{\real^N}\sum\limits_{i,j=1}^{N}\left[y_j\partial_i V\partial_{i}\partial_{j}V+\delta_{ij}\partial_i V \partial_j V \right]\,dy\\
&=\frac{1}{2}\int_{\real^N}y\cdot\nabla(|\nabla V|^2)\,dy+\int_{\real^N}|\nabla V|^2\,dy\\
&=-\frac{N-2}{2}\int_{\real^N}|\nabla V|^2\,dx.
\end{split}
\end{equation}
The second term obtained after this multiplication and integration is the following:
\begin{equation}\label{Poh2}
\alpha\int_{\real^N}V^{1/m}y\cdot\nabla V\,dy=\frac{m\alpha}{m+1}\int_{\real^N}y\cdot\nabla V^{(m+1)/m}=-\frac{mN\alpha}{m+1}\int_{\real^N}V^{(m+1)/m}\,dy.
\end{equation}
The third term gives
\begin{equation}\label{Poh3}
\beta\int_{\real^N}(y\cdot\nabla V^{1/m})(y\cdot\nabla V)\,dy=\frac{\beta}{m}\int_{\real^N}V^{(1-m)/m}(y\cdot\nabla V)^2\,dy
\end{equation}
Finally, the fourth term in \eqref{eq.sim.var}, multiplied by $y\cdot\nabla V$, leads to
\begin{equation}\label{Poh4}
\begin{split}
-\int_{\real^N}|y|^{\sigma}V^{p/m}(y\cdot\nabla V)\,dy&=-\frac{m}{m+p}\int_{\real^N}|y|^{\sigma}y\cdot\nabla V^{(m+p)/m}\,dy\\
&=\frac{m(N+\sigma)}{m+p}\int_{\real^N}|y|^{\sigma}V^{(m+p)/m}\,dy.
\end{split}
\end{equation}
Gathering the identities \eqref{Poh1}, \eqref{Poh2}, \eqref{Poh3} and \eqref{Poh4}, we derive from \eqref{eq.sim.var} the following equality:
\begin{equation}\label{Poh5}
\begin{split}
0&=-\frac{N-2}{2}\int_{\real^N}|\nabla V(y)|^2\,dy-\frac{mN\alpha}{m+1}\int_{\real^N}V^{(m+1)/m}(y)\,dy\\
&+\frac{\beta}{m}\int_{\real^N}V^{(1-m)/m}(y)(y\cdot\nabla V(y))^2\,dy+\frac{m(N+\sigma)}{m+p}\int_{\real^N}|y|^{\sigma}V^{(m+p)/m}(y)\,dy.
\end{split}
\end{equation}
Let us next observe that all the previous integrals are convergent at infinity, as implied by \eqref{decays2}, and thus all the previous calculations are allowed. Indeed, we have
$$
|(y\cdot\nabla V(y))\Delta V(y)|\leq C_4|y|^{-(m+1)(\sigma+2)/(p-m)}, \qquad |\nabla V(y)|^2\leq C_5|y|^{-2m(\sigma+2)/(p-m)-2},
$$
then
\begin{equation*}
\begin{split}
&V^{(m+1)/m}(y)\leq C_6|y|^{-(m+1)(\sigma+2)/(p-m)}, \\ &|V^{(1-m)/m}(y)(y\cdot\nabla V(y))^2|\leq C_8|y|^{-(m+1)(\sigma+2)/(p-m)},
\end{split}
\end{equation*}
and finally
$$
|y|^{\sigma}V^{(m+p)/m}(y)\leq C_9|y|^{-(m+p)(\sigma+2)/(p-m)+\sigma},
$$
for constants $C_i>0$. And, since we are interested on the interval $m<p<p_F(\sigma)$, it follows by direct calculation that
$$
\frac{(m+1)(\sigma+2)}{p-m}>N, \qquad \frac{(m+p)(\sigma+2)}{p-m}-\sigma>N, \qquad \frac{2m(\sigma+2)}{p-m}+2>N,
$$
the last two inequalities holding in fact true for any $p\in(m,p_s(\sigma))$.

\medskip

\noindent \textbf{Step 2. Multiply \eqref{eq.sim.var} by $V$}. After a multiplication by $V$ of the equation \eqref{eq.sim.var} and an integration over $\real^N$ (employing integration by parts when needed), one gets the following identity
\begin{equation}\label{Poh6}
\begin{split}
0&=\int_{\real^N}|\nabla V(y)|^2\,dy+\alpha\int_{\real^N}V^{(m+1)/m}(y)\,dy\\
&-\frac{N\beta}{m+1}\int_{\real^N}V^{(m+1)/m}(y)\,dy-\int_{\real^N}|y|^{\sigma}V^{(m+p)/m}(y)\,dy.
\end{split}
\end{equation}
Once more, these integrals are convergent at infinity, as it follows from \eqref{decays2}.

\medskip

\noindent \textbf{Step 3. Gathering the identities \eqref{Poh5} and \eqref{Poh6}}. We next combine the two identities \eqref{Poh5} and \eqref{Poh6} with the goal of reducing the terms in $|y|^{\sigma}V^{(m+p)/m}(y)$ between them. To this end, we multiply \eqref{Poh6} by $m(N+\sigma)/(m+p)$ and sum the outcome to \eqref{Poh5}. We obtain
\begin{equation}\label{Pohozaev}
\begin{split}
0&=-\frac{N-2}{2}\int_{\real^N}|\nabla V(y)|^2\,dy+\frac{m(N+\sigma)}{m+p}\int_{\real^N}|\nabla V(y)|^2\,dy\\
&-\frac{\alpha mN}{m+1}\int_{\real^N}V^{(m+1)/m}(y)\,dy+\frac{\alpha m(N+\sigma)}{m+p}\int_{\real^N}V^{(m+1)/m}(y)\,dy\\
&-\frac{N\beta m(N+\sigma)}{(m+p)(m+1)}\int_{\real^N}V^{(m+1)/m}(y)\,dy+\frac{\beta}{m}\int_{\real^N}V^{(1-m)/m}(y)(y\cdot\nabla V(y))^2\,dy\\
&=\frac{m(N+2\sigma+2)-p(N-2)}{2(m+p)}\int_{\real^N}|\nabla V(y)|^2\,dy+\frac{\beta}{m}\int_{\real^N}V^{(1-m)/m}(y)(y\cdot\nabla V(y))^2\,dy\\
&+\frac{mQ(m,N,p,\sigma)}{(m+1)(m+p)L}\int_{\real^N}V^{(m+1)/m}(y)\,dy,
\end{split}
\end{equation}
where $L=\sigma(m-1)+2(p-1)>0$ and
\begin{equation}\label{interm18}
Q(m,N,p,\sigma)=(m+1)\sigma^2+(m+1)(N+2)\sigma+N(mN+2)-N(N+2\sigma+2)p.
\end{equation}
The identity \eqref{Pohozaev} will be employed at the end of this section in order to establish non-existence of self-similar solutions for $m<p\leq p_F(\sigma)$.

\medskip

\noindent \textbf{Remark.} Coming back to our self-similar solutions $f(\xi)$, with $\xi=|y|$ in the previous notation, that enter the critical point $Q_1$, they fulfill the decay rates \eqref{decays1}, as follows from \eqref{beh.Q1} and \eqref{beh.Q1.deriv}. Thus, the identity \eqref{Pohozaev} applies in particular to them.

\subsection{A system of planes and surfaces}\label{subsec.geom}

Let us consider now a geometric construction in the phase space associated to the system \eqref{PSsyst} combining several planes and surfaces and generating an invariant region for $p_F(\sigma)<p<p_s(\sigma)$. This system has been already introduced in \cite[Section 6]{IS23b}, thus we shall skip some calculations in the forthcoming steps.

\medskip

\noindent \textbf{1. A plane in the half-space $\{Y>0\}$}. We begin with the plane given by
\begin{equation}\label{plane1}
Z={\rm pln}(X,Y):=(N+\sigma)\left(\frac{X}{N}-Y\right), \qquad Y>0.
\end{equation}
Taking the normal direction as
$$
\overline{n}=\frac{1}{N}(-(N+\sigma),N(N+\sigma),N),
$$
the direction of the flow of the system \eqref{PSsyst} across the plane \eqref{plane1} is given by the sign of
\begin{equation}\label{flow.plane}
F_1(X,Y)=-p(N+\sigma)Y\left[Y-\frac{(p-m)(N+\sigma)+L}{Np(\sigma+2)}X\right].
\end{equation}
Since $Z\geq0$, we have $Y\leq X/N$ on the plane \eqref{plane1}, hence
$$
Y-\frac{(p-m)(N+\sigma)+L}{Np(\sigma+2)}X\leq\frac{X}{N}\left[1-\frac{(p-m)(N+\sigma)+L}{p(\sigma+2)}\right]=\frac{p_F(\sigma)-p}{p(\sigma+2)}X<0,
$$
which proves that $F_1(X,Y)>0$ if $p>p_F(\sigma)$ and $Y>0$.

\medskip

\noindent \textbf{2. A surface in the half-space $\{Y\leq0\}$}. Set
\begin{equation}\label{surface}
Z={\rm sup}(X,Y):=(N+\sigma)\left(\frac{X}{N}-Y\right)-pY^2+\frac{(p-m)(N+\sigma)}{N(\sigma+2)}XY,
\end{equation}
for $Y\leq0$. Taking
$$
\overline{N}=\left(-\frac{(N+\sigma)(\sigma+2+(p-m)Y)}{N(\sigma+2)},-\frac{(p-m)(N+\sigma)}{N(\sigma+2)}X+2pY+N+\sigma,1\right),
$$
as normal vector, we obtain that the direction of the flow of the system \eqref{PSsyst} across the surface \eqref{surface} is given by the sign of the expression
\begin{equation}\label{flow.surf}
\begin{split}
F_2(X,Y)&=\left(Y+\frac{\sigma+2}{p-m}\right)\left[\frac{\sigma(p-m)^2(N+\sigma)}{N^2(\sigma+2)^2}X^2+\frac{(p-m)C(m,N,p,\sigma)}{N(\sigma+2)}XY\right.\\
&\left.+(p-m)pY^2\right],
\end{split}
\end{equation}
with
$$
C(m,N,p,\sigma)=(N+\sigma)(m-1)-2p\sigma.
$$
Observe that all the terms in the right hand side of \eqref{flow.surf} are positive in the region $\{-(\sigma+2)/(p-m)<Y<0\}$, provided $C(m,N,p,\sigma)<0$, that is,
\begin{equation}\label{limitsurf}
p>\frac{(N+\sigma)(m-1)}{2\sigma}, \qquad {\rm or \ equivalently}, \qquad \sigma>\frac{N(m-1)}{2p-m+1}.
\end{equation}
Thus, if \eqref{limitsurf} holds true, $F_2(X,Y)>0$ in $\{-(\sigma+2)/(p-m)<Y<0\}$.

\medskip

\noindent \textbf{3. The plane $\{Y=-(\sigma+2)/(p-m)\}$}. According to Lemma \ref{lem.noret}, if either $N\in\{1,2\}$ or $N\geq3$ and $p\in(m,p_c(\sigma)]$, any trajectory lying on the unstable manifold of $P_0$, once crossing this plane, will not return and will enter the critical point (stable node) $Q_3$.

\medskip

\noindent \textbf{4. The cylinder \eqref{iso1}}. According to Lemma \ref{lem.noret}, if $N\geq3$ and $p\in(p_c(\sigma),p_s(\sigma))$, any trajectory lying on the unstable manifold of $P_0$ will remain in the exterior region to the cylinder \eqref{iso1} and, once crossing the plane $\{Y=-(\sigma+2)/(p-m)\}$, will not return and will again enter the stable node $Q_3$.

\medskip

This construction will be employed next in order to prove non-existence of solutions provided $p_F(\sigma)<p<p_s(\sigma)$.

\subsection{Proof of Theorem \ref{th.nonexist}}\label{subsec.non}

This part is devoted to the proof of Theorem \ref{th.nonexist}. It will be done in a few steps, according to the range of $p$.

\medskip

\noindent \textbf{Step 1: Exploiting the Pohozaev identity.} In this step, we employ the Pohozaev identity \eqref{Pohozaev} constructed in Section \ref{subsec.Poh}. Notice that since either $N\in\{1,2\}$ or $N\geq3$ and $m<p<p_s(\sigma)$, the coefficients of the first two integral terms in \eqref{Pohozaev} are positive. Since all the integrands are also positive, it suffices in \eqref{interm18} to have
$$
Q(m,N,p,\sigma)\geq 0,
$$
in order to ensure that there are no nontrivial solutions to Eq. \eqref{eq.sim.var} and thus no self-similar solutions to Eq. \eqref{eq1}. This condition is equivalent to
\begin{equation}\label{nonex.1}
m\leq p\leq p_1(\sigma):=m+\frac{(\sigma+2)[\sigma(m+1)-N(m-1)]}{N(N+2\sigma+2)},
\end{equation}
and we notice that the right hand side in \eqref{nonex.1} is bigger than $m$ provided $\sigma>N(m-1)/(m+1)$ and satisfies
$$
m+\frac{(\sigma+2)[\sigma(m+1)-N(m-1)]}{N(N+2\sigma+2)}-p_F(\sigma)=\frac{(\sigma+2)[\sigma(m-1)-mN-2]}{N(N+2\sigma+2)},
$$
which is non-negative if $\sigma\geq\sigma^*=(mN+2)/(m-1)$, and this is how the exponent $\sigma^*$ in \eqref{sigmastar} comes into play. We thus conclude that, if
\begin{equation}\label{cond1}
\sigma\geq\sigma^*, \qquad m<p\leq p_F(\sigma),
\end{equation}
there are no non-trivial self-similar solutions (not even in non-radial form) to Eq. \eqref{eq1}. Moreover, let us notice that
$$
\frac{(m+1)(\sigma+2)}{p_1(\sigma)-m}-N=\frac{N[2mN+2m+2+\sigma(m+1)]}{\sigma(m+1)-N(m-1)}>0,
$$
provided $\sigma>\sigma^*>N(m-1)/(m+1)$ (as follows from \eqref{interm23} below), which implies that the integrals leading to the proof of the identity \eqref{Pohozaev} are also convergent at infinity for $p\in(m,p_1(\sigma))$ and thus the identity holds true.

\medskip

\noindent \textbf{Step 2: $p_F(\sigma)<p<p_s(\sigma)$}. In this part, we employ the system of planes and surfaces constructed in Section \ref{subsec.geom}. Indeed, the Taylor expansion near the origin of the unstable manifold stemming from $P_0$, according to the theory in \cite[Section 2.7]{Shilnikov}, has the form
\begin{equation}\label{unstable}
Z=(N+\sigma)\left(\frac{X}{N}-Y\right)+aX^2+bXY+cY^2+o(|(X,Y)|^2),
\end{equation}
with coefficients
\begin{equation}\label{intermXX}
\begin{split}
&a=-\frac{\sigma(N+\sigma)A(m,N,p,\sigma)}{N^2(\sigma+2)(N+2)(N+\sigma+2)(N+2\sigma+2)}, \\
&b=-\frac{(N+\sigma)A(m,N,p,\sigma)}{N(\sigma+2)(N+\sigma+2)(N+2\sigma+2)}, \qquad c=-\frac{(N+\sigma)p}{N+2\sigma+2},
\end{split}
\end{equation}
where
\begin{equation*}
\begin{split}
A(m,N,p,\sigma)&=-(N^2+3N\sigma+4N+2\sigma+4)(p-p_F(\sigma))\\&+\frac{(\sigma+2)(N+2)(N+2\sigma+2)}{N}.
\end{split}
\end{equation*}
Noticing that, with the coefficients defined in \eqref{intermXX}, we have on the plane $Y=X/N$ that
$$
aX^2+bXY+cY^2\Big|_{Y=X/N}=\frac{(N+\sigma)(p-p_F(\sigma))}{N(N+2)(\sigma+2)}X^2>0,
$$
for $p>p_F(\sigma)$, we infer that the orbits stemming from $P_0$ for $\sigma>0$ enter the region $\mathcal{F}=\{(X,Y,Z): Y>0, Z>{\rm pln}(X,Y)\}$ and according to the direction of the flow \eqref{flow.plane} of the system \eqref{PSsyst} across the plane \eqref{plane1}, they stay in the region $\mathcal{F}$ until intersecting the plane $\{Y=0\}$. For more details of the previous calculations, the reader is referred to \cite[Section 6]{IS23b}. When these orbits intersect the plane $\{Y=0\}$, the surface \eqref{surface} and, if $p>p_c(\sigma)$, also the cylinder \eqref{iso1} enter into play. More precisely, introducing the region
\begin{equation}\label{regionG}
\mathcal{G}:=\left\{(X,Y,Z)\in\real^3: Y\leq0, \ Z>{\rm sup}(X,Y), \ Z>-\frac{N+\sigma}{N-2}Y(mY+N-2)\right\},
\end{equation}
we deduce that the trajectories emanating from $P_0$ on its two-dimensional unstable manifold pass from the region $\mathcal{F}$ into the region $\mathcal{G}$ as crossing the plane $\{Y=0\}$, and according to the direction of the flow of the system \eqref{PSsyst} on its walls given by \eqref{flow.cyl} and \eqref{flow.surf} if \eqref{limitsurf} is in force (and also on the region contained in $\{Y=0\}$ equal to $\mathcal{F}\cap\mathcal{G}$, since $X-Z<0$ while $Z>{\rm pln}(X,0)$), these trajectories will stay in the region $\mathcal{G}$ at least until crossing the plane $\{Y=-(\sigma+2)/(p-m)\}$. Lemma \ref{lem.noret} then entails that, if a trajectory escapes $\mathcal{G}$ by crossing the latter plane, it will connect to $Q_3$ and not to $Q_1$.

\medskip

It thus remains to show that there is no way to enter the critical point $Q_1$ from the region $\mathcal{G}$ defined in \eqref{regionG}. To this end, we recall that the center manifolds on which any trajectory might reach $Q_1$ have the local Taylor expansion \eqref{cmf} in the variables $(x,y,z)$ introduced in \eqref{change2}. In order to compare, we have to pass the region $\mathcal{G}$ into the same variables, and in particular, we infer from \eqref{surface} and \eqref{regionG} that
\begin{equation}\label{interm19}
-py^2+(N+\sigma)\left[-x+\frac{p-m}{N(\sigma+2)}\right]y+\frac{N+\sigma}{N}x-xz<0,
\end{equation}
or equivalently, by solving \eqref{interm19} in terms of $y$ and performing a Taylor development up to the second order,
\begin{equation}\label{cmf1}
\begin{split}
y<y_{\rm sup}&:=-\frac{\sigma+2}{p-m}x+\frac{N(\sigma+2)^2(m(N+\sigma)-p(N-2))}{(N+\sigma)(p-m)^3}x^2\\&+\frac{(\sigma+2)N}{(N+\sigma)(p-m)}xz+o(|(x,z)|^2)
\end{split}
\end{equation}
in a neighborhood of $(x,y,z)=(0,0,0)$. We also infer from \eqref{cmf} that, on the trajectories entering $Q_1$, one has
\begin{equation}\label{cmf2}
\begin{split}
y=y(Q_1)&:=-\frac{\sigma+2}{p-m}x+\frac{(\sigma+2)^2(m(N+\sigma)-p(N-2))}{(p-m)^3}x^2\\&+\frac{\sigma+2}{p-m}xz+o(|(x,z)|^2).
\end{split}
\end{equation}
Notice next that \eqref{cmf1} and \eqref{cmf2} imply
\begin{equation}\label{interm20}
\begin{split}
y(Q_1)-y_{\rm sup}&=\frac{\sigma x}{N+\sigma}\left[\frac{(\sigma+2)^2(m(N+\sigma)-p(N-2))}{(p-m)^3}x+\frac{\sigma+2}{p-m}z\right]+o(|(x,z)|^2)\\
&=\frac{\sigma(\sigma+2)x^2}{(N+\sigma)(p-m)}(Z-Z(P_2))+o(|(x,z)^2|)>0,
\end{split}
\end{equation}
where we have used the fact that $Z=z/x$, as anyway $Z(P_2)\leq0$ if $p_F(\sigma)<p\leq p_c(\sigma)$ according to \eqref{zp2} and $Z>Z(P_2)$ inside the region $\mathcal{G}$ according to the cylinder \eqref{iso1}, if $p_c(\sigma)<p<p_s(\sigma)$. We then derive from \eqref{interm20} that $y<y_{\rm sup}<y(Q_1)$ in a neighborhood of $Q_1$ and inside the region $\mathcal{G}$, for any $\sigma>0$, which shows that no trajectories can enter the critical point $Q_1$ throughout $\mathcal{G}$. This, together with Lemma \ref{lem.noret}, establish the non-existence of trajectories connecting $P_0$ to $Q_1$ for $p_F(\sigma)<p<p_s(\sigma)$, provided \eqref{limitsurf} is in force.

\medskip

\noindent \textbf{Step 3: End of the proof}. On the one hand, we have established in Step 1 of the current proof that non-existence of self-similar blow-up profiles to Eq. \eqref{eq1} holds true for $m<p\leq p_F(\sigma)$, provided $\sigma>(mN+2)/(m-1)$. But in reality, we have a bit more: provided $\sigma>N(m-1)/(m+1)$, we have established in fact non-existence for
$$
m<p\leq p_1(\sigma)=m+\frac{(\sigma+2)[\sigma(m+1)-N(m-1)]}{N(N+2\sigma+2)},
$$
and if $\sigma>\sigma^*$ we have $p_1(\sigma)>p_F(\sigma)$, which is also needed in order to match this range with the outcome of the second step (that cannot be improved below the exponent $p_F(\sigma)$). On the other hand, we have covered in Step 2 the range $p_F(\sigma)<p<p_s(\sigma)$ if \eqref{limitsurf} holds true, which in particular involves another dependence of $p$ in terms of $\sigma$, namely
$$
p>p_2(\sigma)=\frac{(N+\sigma)(m-1)}{2\sigma}.
$$
The proof is completed if we show that, at least for $\sigma$ sufficiently large, we have $p_2(\sigma)<p_1(\sigma)$. We compute
$$
p_1(\sigma)-p_2(\sigma)=\frac{(N^2+2N\sigma+2\sigma^2+2N+4\sigma)(\sigma(m+1)-N(m-1))}{2N(N+2\sigma+2)\sigma}
$$
and observe that $p_1(\sigma)>p_2(\sigma)$ if $\sigma>N(m-1)/(m+1)$. Moreover, since
\begin{equation}\label{interm23}
\frac{mN+2}{m-1}-\frac{N(m-1)}{m+1}=\frac{N(3m-1)+2(m+1)}{(m-1)(m+1)}>0,
\end{equation}
we infer that the whole interval $p\in(m,p_s(\sigma))$ is covered, provided $\sigma>\sigma^*$ with $\sigma^*$ introduced in \eqref{sigmastar}, completing the proof of Theorem \ref{th.nonexist}.

We end this proof with a plot in Figure \ref{fig3} of the result of a numerical experiment suggesting how all the trajectories $l_C$ cross the no-return plane $\{Y=-(\sigma+2)/(p-m)\}$ (according to Lemma \ref{lem.noret}), as a visual image of how non-existence takes place.

\begin{figure}[ht!]
  \begin{center}
  \includegraphics[width=11cm,height=7.5cm]{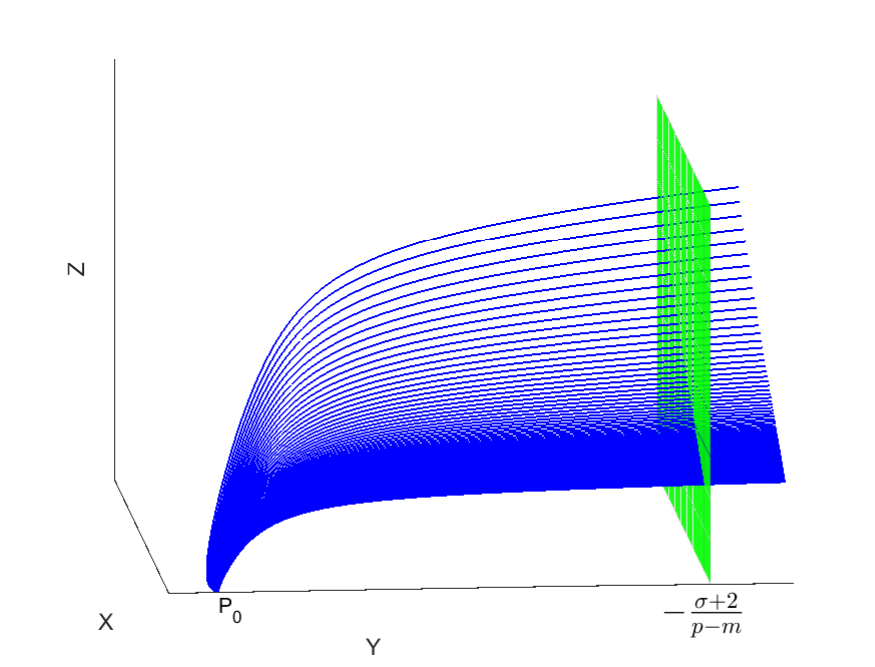}
  \end{center}
  \caption{Trajectories $l_C$ crossing the plane $\{Y=-(\sigma+2)/(p-m)\}$ in the non-existence range. Experiments for $m=2$, $N=5$, $p=2.1$, $\sigma=1$}\label{fig3}
\end{figure}

\bigskip

\noindent \textbf{Acknowledgements} R. G. I. and A. S. are partially supported by the Project PID2020-115273GB-I00 and by the Grant RED2022-134301-T funded by MCIN/AEI/10.13039/ \\ 501100011033 (Spain).

\bigskip

\noindent \textbf{Data availability} Our manuscript has no associated data.

\bigskip 

\noindent \textbf{Conflict of interest} The authors declare that there is no conflict of interest.

\bibliographystyle{plain}

\end{document}